\documentclass[reqno,11pt]{amsart}
\usepackage{amsfonts, amsmath, amssymb, amsthm, mathtools}
\usepackage[margin=2.5cm, heightrounded]{geometry}
\usepackage{fancybox}
\usepackage{hhline, float}
\usepackage{mathrsfs}
\usepackage{nccmath}
\usepackage{bm}
\usepackage[normalem]{ulem}
\usepackage{cite}
\usepackage[dvipsnames]{xcolor}
\usepackage[colorlinks=true]{hyperref}
\hypersetup{citecolor=MidnightBlue, linkcolor=MidnightBlue, urlcolor=MidnightBlue}

\newcommand{\dt}{\textup{d}t}
\newcommand{\ds}{\textup{d}s}
\newcommand{\dx}{\textup{d}x}
\newcommand{\dy}{\textup{d}y}

\usepackage{enumitem}
\usepackage{microtype}
\usepackage[nameinlink,noabbrev]{cleveref}
\usepackage{booktabs}
\usepackage{array}
\usepackage{longtable}
\newtheorem{thm}{Theorem}[section]
\newtheorem{lemma}[thm]{Lemma}
\newtheorem{prop}[thm]{Proposition}

\theoremstyle{definition}

\newcommand{\R}{\mathbb R}

\newcommand{\loc}{\mathrm{loc}}

\allowdisplaybreaks
\numberwithin{equation}{section}

\makeatletter
\@namedef{subjclassname@2020}{%
\textup{2020} Mathematics Subject Classification}
\newcommand{\maketitlewithmetadata}{%
  \let\savedsubjclass\@subjclass
  \let\savedkeywords\@keywords
  \let\@subjclass\@empty
  \let\@keywords\@empty
  \maketitle
  \begingroup
    \small
    \noindent\textit{\subjclassname.}\enspace\savedsubjclass\par
    \smallskip
    \noindent\textit{\keywordsname.}\enspace\savedkeywords\par
  \endgroup
}
\makeatother

\begin{document}

\title{Classification of positive solutions to the critical Lane--Emden system}

\author{Yuxia Guo}
\address[Yuxia Guo]{Department of Mathematical Sciences, Tsinghua University, Beijing, P.R. China}
\email{yguo@tsinghua.edu.cn}

\author{Yichen Hu}
\address[Yichen Hu]{School of Mathematical Sciences, Dalian University of Technology, Dalian, Liaoning, P. R. China}
\email{huyc24@dlut.edu.cn}

\author{Shaolong Peng}
\address[Shaolong Peng]{School of Mathematical Sciences, Beihang University (BUAA), Beijing, P.R. China}
\email{shaolongpeng@buaa.edu.cn}

\author{Tingfeng Yuan}
\address[Tingfeng Yuan]{Department of Mathematical Sciences, Tsinghua University, Beijing, P. R. China}
\email{ytf22@mails.tsinghua.edu.cn}

\thanks{Yuxia Guo was supported by the National Natural Science Foundation of China (No. 12271283) and National Key R\&D Program (2023YFA1010002). Yichen Hu is supported by NNSF of China (No. 12501135) and the Fundamental Research Funds for the Central Universities (DUT24RC(3)110). Shaolong Peng was supported by the National Natural Science Foundation of China (No. 12401148 and 12571113), and Beijing Natural Science Foundation (No. 1262014).}

\begin{abstract}
We classify all positive classical solutions of the Lane--Emden system
\begin{equation*}
    \begin{cases}
        -\Delta u=v^p & \text{in }\R^N,\\
        -\Delta v=u^q & \text{in }\R^N,
    \end{cases}
\end{equation*}
where $N\geq3$ and $p,q>0$ satisfy
\begin{equation*}
    \frac{1}{p+1}+\frac{1}{q+1}=\frac{N-2}{N}.
\end{equation*}
Without assuming boundedness, decay at infinity, or finite energy, we prove that every such solution has the form
\begin{equation*}
    \begin{cases}
        u(x)=\lambda^{\frac{N}{q+1}}U(\lambda(x-a)),\\
        v(x)=\lambda^{\frac{N}{p+1}}V(\lambda(x-a)),
    \end{cases}
\end{equation*}
for some $a\in\R^N$ and $\lambda>0$, where $(U,V)$ is the positive radial ground state centered at the origin and normalized by $U(0)=1$. Consequently, both components are radially symmetric about the same point and satisfy
\begin{equation*}
    \int_{\R^N}u^{q+1}\,\dx=\int_{\R^N}v^{p+1}\,\dx<\infty.
\end{equation*}
Our result extends the classical classification of positive solutions of the critical Sobolev equation by Caffarelli, Gidas and Spruck \cite{CGS1989} and Chen and Li \cite{CL1991} to the critical Lane--Emden system.
\end{abstract}

\date{}
\maketitlewithmetadata

\section{Introduction}\label{sec:introduction}

We consider the  positive classical solutions for the following Lane--Emden system:
\begin{equation}\label{Main1}
    \begin{cases}
        -\Delta u=v^p & \text{in }\R^N,\\
        -\Delta v=u^q & \text{in }\R^N,
    \end{cases}
\end{equation}
where $N\geq3$, and the exponents $p,q>0$ lie on the critical hyperbola:
\begin{equation}\label{critical-hyperbola}
    \frac{1}{p+1}+\frac{1}{q+1}=\frac{N-2}{N}.
\end{equation}
Interchanging the two components if necessary, we assume throughout the paper  that $p\leq q$. In particular,
\begin{equation*}
    \frac{2}{N-2}<p\leq\frac{N+2}{N-2}\leq q, \ \  pq>1.
\end{equation*}
Our purpose is to classify these solutions without assuming boundedness, decay at infinity, or finite energy.

\subsection{Research Background} The Lane--Emden equation comes from the study of self-gravitating gaseous spheres in stellar structure, beginning with Lane's work on the temperature of the Sun \cite{Lane1870} and developed systematically in Emden's monograph \cite{Emden1907}. In its classical radial form, the scalar equation reads as:
\begin{equation*}
    w''(r)+\frac{N-1}{r}w'(r)+w(r)^s=0,
\end{equation*}
with $N=3$ in the stellar model. System \eqref{Main1} is a coupled elliptic extension of this scalar equation. For the critical system considered here, its variational and integral structures provide the immediate mathematical motivation.

The critical hyperbola is also the boundary of the subcritical range in the Lane--Emden conjecture. For the scalar equation $-\Delta w=w^s$ in $\R^N$, Gidas and Spruck \cite{GS1981} proved nonexistence of positive classical solutions when $1<s<(N+2)/(N-2)$, without a condition at infinity. For the coupled system, the subcritical condition is
\begin{equation*}
    \frac{1}{p+1}+\frac{1}{q+1}>\frac{N-2}{N}.
\end{equation*}
Souplet \cite{Souplet2009} proved the Lane--Emden conjecture in dimension four and obtained partial nonexistence results in higher dimensions. Recently, OpenAI \cite{OpenAI} resolved the full subcritical Lane-Emden conjecture by developing a new method based on source-layer interval comparisons. This method provides a way to measure the failure of radial symmetry without assuming any decay at infinity. This naturally raises the question of the classification at the critical hyperbola, where positive entire solutions do exist. Our theorem shows that every positive classical entire solution at the critical hyperbola belongs to the ground-state family, without any assumption on its behavior at infinity.

For scalar quasilinear equations and inequalities, Serrin and Zou \cite{SZ2002} proved Cauchy--Liouville theorems and universal bounds, including results in the whole space and exterior domains. Phuc and Verbitsky \cite{PV2008} used nonlinear potential theory to study solvability, estimates, and removable singularities for scalar quasilinear and Hessian equations of Lane--Emden type with measure data.

To explain the role of the  critical hyperbola \eqref{critical-hyperbola}, we set
\begin{equation}\label{scaling-exponents}
    \alpha=\frac{N}{q+1},\ \  \beta=\frac{N}{p+1}.
\end{equation}
The critical relation gives that
\begin{equation*}
    \alpha+\beta=N-2,\ \  \alpha+2=p\beta,\ \  \beta+2=q\alpha.
\end{equation*}
Consequently, one can check that the translations and the dilations
\begin{equation}\label{critical-scaling}
    u_{\lambda,a}(x)=\lambda^\alpha u(\lambda(x-a)),\ \  v_{\lambda,a}(x)=\lambda^\beta v(\lambda(x-a)), \ \  \lambda>0,\ a \in \R^N,
\end{equation}
preserve both the system and the integrals of $u^{q+1}$ and $v^{p+1}$. For sufficiently regular, rapidly decaying pairs, the system is also the Euler--Lagrange system of the Hamiltonian functional
\begin{equation*}
    \mathcal J(u,v)=\int_{\R^N}\nabla u\cdot\nabla v\ \dx -\frac{1}{q+1} \int_{\R^N}u^{q+1}\ \dx -\frac{1}{p+1}\int_{\R^N}v^{p+1}\ \dx.
\end{equation*}
The identity $\alpha+\beta=N-2$ makes all three terms invariant under the translations and the dilations \eqref{critical-scaling}. Thus the hyperbola is the scaling threshold for this variational structure. This is also reflected in Pohozaev-type identities for elliptic systems \cite{Mitidieri1993}. In the scalar Dirichlet problem on bounded domains, Brezis and Nirenberg \cite{BN1983} studied the critical Sobolev equation with lower-order terms, where the failure of compactness of the critical embedding is a central difficulty.

There is a closely related connection with the Hardy--Littlewood--Sobolev (HLS in short) inequality. Let $K(x)=((N-2)|\mathbb S^{N-1}|)^{-1}|x|^{2-N}$. For
\begin{equation*}
    r=\frac{q+1}{q},\ \  s=\frac{p+1}{p},
\end{equation*}
condition \eqref{critical-hyperbola} is equivalent to $1/r+1/s+(N-2)/N=2$, and the Hardy--Littlewood--Sobolev inequality takes the form
\begin{equation*}
    \int_{\R^N}\int_{\R^N} F(x)K(x-y)G(y)\ \dx \dy \leq C\|F\|_{L^r}\|G\|_{L^s}.
\end{equation*}
The Euler--Lagrange equations for nonnegative optimizers, after a normalization of the constants and the substitutions $F=u^q$, $G=v^p$, yield the integral system
\begin{equation*}
    u=K*v^p,\ \  v=K*u^q.
\end{equation*}
Lieb \cite{Lieb1983} proved the existence of optimizers for the sharp Hardy--Littlewood--Sobolev inequality. Under the substitutions above, their integrability gives the natural finite-energy class
\begin{equation}\label{finite-energy-class}
    u\in L^{q+1}(\R^N),\ \  v\in L^{p+1}(\R^N).
\end{equation}

The relation between sharp inequalities and critical equations also appears on other spaces. Beckner \cite{Beckner1993} studied sharp Sobolev inequalities on the sphere and the Moser--Trudinger endpoint. On the Heisenberg group, Jerison and Lee \cite{JL1988} determined the sharp Sobolev inequality and its extremals in connection with the CR Yamabe problem, and Frank and Lieb \cite{FL2012} established sharp Hardy--Littlewood--Sobolev inequalities. These results concern extremals in their respective function spaces and geometries.

\subsection{Known classification results}
When $p=q=(N+2)/(N-2)$ and $u=v=w$, system \eqref{Main1} reduces to
\begin{equation*}
    -\Delta w=w^{\frac{N+2}{N-2}},\  \ \text{in }\R^N.
\end{equation*}
This is the critical Sobolev equation, which is also the equation for a conformally flat metric of constant positive scalar curvature after normalization. The classification of its positive entire solutions, without an a priori finite-energy assumption, is a classical consequence of the symmetry theory of Caffarelli, Gidas, and Spruck \cite{CGS1989}. Chen and Li \cite{CL1991} gave a shorter proof by the method of moving planes. Li and Zhu \cite{LZ1995} used the method of moving spheres to classify nonnegative solutions of conformally invariant scalar equations in a half-space with critical nonlinear boundary conditions. Li and Li \cite{LiLi2005} proved Liouville and Harnack theorems for scalar conformally invariant fully nonlinear equations under suitable structural assumptions. For the scalar critical equation with isolated singularities, Korevaar, Mazzeo, Pacard, and Schoen \cite{KMPS1999} obtained refined asymptotic expansions near the singular points.

For $ 0 < \gamma < N $, the related scalar integral equation is
\begin{equation*}
    w(x)=\int_{\R^N}|x-y|^{\gamma-N} w(y)^{\frac{N+\gamma}{N-\gamma}}\ \dy.
\end{equation*}
Chen, Li, and Ou \cite{CLO2006} classified its positive regular solutions. Li \cite{Li2004} gave a moving-spheres proof of the classification for locally bounded positive solutions and proved regularity under the local assumption $w\in L^{2N/(N-\gamma)}_{\loc}(\R^N)$. As explained in \cite{Li2004}, this answers Lieb's classification question for the scalar HLS equation in that regularity class. The distinction between classifying optimizers and classifying general solutions is also central to the present problem. For unequal exponents, system \eqref{Main1} has two distinct scaling weights and does not reduce to these scalar equations.

For the system \eqref{Main1}, existence of ground states follows from the critical variational theory of Lions \cite{Lions1985}. Hulshof and van der Vorst \cite{HV1996} established uniqueness of the ground state up to translations and dilations and described its asymptotic behavior. We fix a positive radial ground state $(U,V)$ centered at the origin and normalized by $U(0)=1$. The natural homogeneous Sobolev spaces are
\begin{equation*}
    (U,V)\in\dot W^{2,\frac{p+1}{p}}(\R^N) \times\dot W^{2,\frac{q+1}{q}}(\R^N),
\end{equation*}
where $\dot W^{2,t}(\R^N)$ denotes the completion of $ C_c^\infty (\R^N) $ with respect to the norm $\| \Delta \cdot \|_{L^t}$. On the critical hyperbola, the corresponding Sobolev embeddings give \eqref{finite-energy-class}. The decay of $U$ changes at $p=N/(N-2)$, whereas $V$ has the Newtonian decay $|x|^{2-N}$. More precise expansions of the ground state and its derivatives were obtained by Guo, Hu, Peng, and Yuan \cite{GHPY2025}. Frank, Kim, and Pistoia \cite{FKP2021} proved that the ground state is nondegenerate, which says that in the natural linearized class, the kernel consists exactly of the modes generated by translations and dilations.

The finite-energy symmetry theory is also well developed. Chen, Li, and Ou \cite{CLO2005} proved radial symmetry for the associated critical integral system under global integrability assumptions when $ p > 1, q > 1$. Hang \cite{Hang2007} treated the full admissible exponent range, including exponents below one. Together with radial uniqueness, these results classify positive solutions in the finite-energy class. The global integrability hypothesis is essential to the formulation of those theorems: they do not by themselves imply that an arbitrary positive classical solution of \eqref{Main1} satisfies \eqref{finite-energy-class}.

Integral identities provide another way to study critical equations. Ciraolo, Figalli, and Roncoroni \cite{CFR2020} classified positive finite-energy solutions of critical anisotropic $p$-Laplace equations in convex cones with the natural Neumann boundary condition. In a different geometry, Catino, Li, Monticelli, and Roncoroni \cite{CLMR2025} classified all positive solutions of the critical scalar equation on the first Heisenberg group $\mathbb H^1$, without finite-energy or symmetry assumptions. Their result uses the Jerison--Lee identity and integral estimates. These works concern scalar equations; they also show why it is useful to identify precisely when global energy assumptions can be removed.

There are also quantitative results in the finite-energy setting. Figalli and Zhang \cite{FZ2022} proved sharp gradient stability for the Sobolev inequality. Deng, Sun, and Wei \cite{DSW2025} obtained sharp quantitative estimates for Struwe's decomposition, controlling the distance of approximate solutions of the scalar critical equation to sums of bubbles. These results start in a Sobolev energy space. For system \eqref{Main1}, our aim is first to show that every positive classical solution belongs to its natural energy space.

The critical Lane--Emden system arises from the Euler--Lagrange equations for the Hardy--Littlewood--Sobolev inequality with the Newtonian kernel. This variational connection goes back to the work of Lieb \cite{Lieb1983} on sharp constants and optimizers. Our theorem classifies all positive classical  solutions of the associated differential system without any a priori integrability assumption; finite energy follows from the classification. In the scalar case, Chen, Li and Ou \cite{CLO2006} and
Li \cite{Li2004} established classification results for critical integral equations with Riesz kernels. For the Newtonian kernel, our theorem extends this scalar classification to the critical Lane--Emden system.

\subsection{Main result}
Denote the translation and dilation orbit of the normalized ground state by
\begin{equation}\label{ground-state-orbit}
    \mathcal G= \left\{ \left( \lambda^\alpha U( \lambda(\cdot-a)), \lambda^\beta V(\lambda(\cdot-a))\right):  a\in\R^N,\ \lambda>0\right\}.
\end{equation}
Here is our main result.
\begin{thm}\label{thm:main}
    Let $N\geq3$ and let $p,q>0$ satisfy \eqref{critical-hyperbola}. If $(u,v)\in C^2(\R^N)\times C^2(\R^N)$ is a positive solution of \eqref{Main1}, then there exist a unique $a\in\R^N$ and a unique $\lambda>0$ such that
    \begin{equation}\label{main-classification}
        u(x)=\lambda^\alpha U(\lambda(x-a)),\ \   v(x)=\lambda^\beta V(\lambda(x-a)).
    \end{equation}
    Both components are radially symmetric about $a$ and strictly decreasing in $|x-a|$. Their energies satisfy
    \begin{equation}\label{main-energy}
        \int_{\R^N}u^{q+1}\ \dx =\int_{\R^N}v^{p+1} \ \dx  = E_* < \infty, \ \   E_*:=\int_{\R^N}U^{q+1}\ \dx  =\int_{\R^N}V^{p+1} \ \dx.
    \end{equation}
    The pair also belongs to the natural Sobolev spaces:
    \begin{equation*}
        (u,v) \in \dot W^{2,\frac{p+1}{p}}(\R^N) \times \dot W^{2,\frac{q+1}{q}}(\R^N).
    \end{equation*}
\end{thm}

Thus finite energy, boundedness, and decay are conclusions of the classification rather than hypotheses. The normalization $U(0)=1$ fixes the radial profile, and the only remaining parameters are its center and scale.

\subsection{Outline of the proof}

Our proof begins with the integral representation of system \eqref{Main1} via the Newtonian kernel $K$ in $\R^N$. We first establish uniform integral estimates for the source terms $v^p$ and $u^q$ on balls, without assuming boundedness or finite energy. The main task is then to prove that every positive solution of \eqref{Main1} is radially symmetric about a common point. Indeed, once radial symmetry is established, a comparison argument for the corresponding ODE system shows that the solution is unique up to translations and dilations.

To establish radial symmetry, our method is based on the source-layer interval comparison method introduced in \cite{OpenAI}. The key observation is that the midpoint term in the interval comparison has a definite sign. In the critical case, we retain this nonnegative term and use it to define a defect that measures the failure of radial symmetry. We then show that the vanishing of this defect forces both components to be radially symmetric about a common point.

To prove that the defect vanishes, we establish an estimate in terms of sublinear powers of the local energy
\begin{equation*}
    E(R)=\int_{B_R}(u^{q+1}+v^{p+1})\,\dx.
\end{equation*}
The key point is that the critical hyperbola allows the leading terms in the integral identities to cancel, leaving only error terms that can be controlled by the local energy.

For bounded positive solutions, we first show that the condition $E(R)=o(R)$ as $R\to\infty$ forces the defect to vanish. To establish this energy growth condition, we combine a local rigidity result with a covering argument to derive an iterative inequality for $E(R)$. Iterating this inequality yields $E(R)=o(R)$, and hence radial symmetry for all bounded positive solutions. Finally, we apply the doubling lemma \cite{PQS2007} to remove the boundedness assumption and complete the classification.

The paper is organized as follows. In section 2, we establish the integral representation, radial uniqueness, and a local isolation theorem, which states that if a positive normalized solution is sufficiently close to the ground state on a fixed ball, then this solution must coincide with the ground state. In section 3, we  introduce the nonnegative defect and show that its vanishing implies radial symmetry of both components about a common point. Section 4 is devoted to establishing  an estimate for the defect in terms of sublinear powers of the local energy $E(R)$ and proving  that $E(R)=o(R)$ forces the defect to vanish for bounded positive solutions. Finally, In Section 5, we complete the proof of the classification by establishing $E(R)=o(R)$ through an iterative argument and removing the boundedness assumption.

\section{Local isolation theorem}\label{sec:preliminaries}

Recall that the system \eqref{Main1} admits a radial ground state $ (U,V) \in \dot{W}^{2,\frac{p+1}{p}} (\R^N) \times \dot{W}^{2,\frac{q+1}{q}} (\R^N) $ with $ U(0) = 1 $. In this section, we prove the local isolation theorem of the ground state $(U,V)$, which says that if a positive normalized solution is close enough to $(U,V)$ in a compact set, then this solution must be  $(U,V)$.

We first recall some known results for $(U,V)$. The decay rate of $ (U, V) $ was first established in \cite{HV1996} and later refined in \cite{GHPY2025}.

\begin{thm}[Hulshof--Van der Vorst  \cite{HV1996}]\label{Decay-GS}
    Assume $(p,q)$ satisfies \eqref{critical-hyperbola} and $ p \leq q $. Then there are two constants $a, b >0$, such that the ground state $ ( U , V )$ of \eqref{Main1} with $U(0) =1$ satisfies
    \begin{align*}
        \lim_{r \to +\infty} r^{N-2} V(r) = a , \ \
        \begin{cases}
            \lim\limits_{r \to +\infty} r^{N -2} U(r) = b, \ \ & \text{if} \ \ \frac{N}{N-2} < p \leq \frac{N+2}{N-2}, \\
            \lim\limits_{r \to +\infty} \frac{r^{N -2}}{\log r} U(r) = b, \ \ & \text{if} \ \ p = \frac{N}{N-2}, \\
            \lim\limits_{r \to +\infty} r^{(N-2)p -2 } U(r) = b , \ \ & \text{if} \ \ \frac{2}{N-2} < p < \frac{N}{N-2}.
        \end{cases}
    \end{align*}
    Furthermore, it holds that
    \begin{align*}
        a^p = b(p( N^2 - 4) - 2N - p^2 (N-2)^2), \ \ \text{when} \ \ \frac{2}{N-2} < p < \frac{N}{N-2}.
    \end{align*}
\end{thm}

The ground state $ (U, V) $ is non-degenerate, which means that among solutions vanishing at infinity, the kernel of the linearized operator of \eqref{Main1} at $ (U, V) $ is spanned by the translation and dilation derivatives and has dimension $N+1$.

\begin{thm}[Frank--Kim--Pistoia \cite{FKP2021}]\label{non-degeneracy}
    Suppose $(U,V)$ is a ground state of \eqref{Main1}. Then all classical solutions $(\phi, \psi)$ of the system
    \begin{align*}
        \begin{cases}
            - \Delta \phi = p V^{p-1} \psi, \ \ & \text{in} \ \ \R^N, \\
            - \Delta \psi = q U^{q-1} \phi, \ \ & \text{in} \ \ \R^N,
        \end{cases}
    \end{align*}
    with $ |\phi| + |\psi| \to 0$ as $|x| \to +\infty$ are linear combinations of $(\partial_{x_i}U, \partial_{x_i} V)$, $ i =1,2,\cdots, N$ and
    \begin{align*}
        \left( x \cdot \nabla U + \frac{N}{q+1} U, \ \   x \cdot \nabla V + \frac{N}{p+1} V \right).
    \end{align*}
\end{thm}

The decay and integrability properties of the ground state are well known. For an arbitrary positive solution $(u,v)$ of \eqref{Main1}, however, boundedness, decay, and finite energy are not assumed. We first establish an integral representation.

 Let $K$ be the fundamental solution of $-\Delta$ in $\R^N$, given by
\begin{align*}
    K(x) = \frac{1}{(N - 2) |\mathbb{S}^{N-1}| \cdot |x|^{N-2}}.
\end{align*}
Whenever the integral is well defined, we use $ \mathcal{I}(h)$ to denote $$K * h(x) = \int_{\R^N} K(x-y) h(y) \ \dy.$$

\begin{prop}\label{lem-int-repre}
    Suppose $(u,v)$ is a positive solution of \eqref{Main1}. Then, for all $a\in\R^N$ and $r>0$, we have
    \begin{align}\label{int-repre-uv}
        u(x) = \int_{\R^N} K(x-y) v^p(y) \ \dy = \mathcal{I} (v^p)(x), \ \ v(x) =\int_{\R^N} K(x-y) u^q(y) \ \dy = \mathcal{I} (u^q)(x),
    \end{align}
    and
    \begin{align}\label{int-uv-bound}
        \int_{B_r(a)} u^q \leq C r^{\alpha}, \ \ \int_{B_r(a)} v^p \leq C r^{\beta},
    \end{align}
    where $\alpha = \frac{N}{q+1}$, $ \beta = \frac{N}{p+1} $. Moreover, we have
    \begin{align}\label{int-gradient-uv}
         \nabla u(x) = \int_{\R^N} \nabla K(x-y) v^p(y) \ \dy, \ \  \nabla v(x) =\int_{\R^N}\nabla  K(x-y) u^q(y) \ \dy.
    \end{align}
    Furthermore, if $u$ and $v$ are bounded in $\R^N$, then
    \begin{align}\label{gradient-uv-bound}
        ||\nabla u ||_{L^{\infty}} \leq C(||v||_{ L^{\infty} }^p + ||u||_{L^{\infty}} ), \ \ ||\nabla v ||_{L^{\infty}} \leq C(||u||_{ L^{\infty} }^q + ||v||_{L^{\infty}} ).
    \end{align}
\end{prop}

\begin{proof}
    We first prove that the convolution terms in \eqref{int-repre-uv} are well-defined. Let $w$ be a nonnegative classical solution of $ -\Delta w = h \geq 0 $ in $ \mathbb R^N $. We claim that $ \mathcal{I}h $ is finite everywhere and satisfies $\mathcal{I} h \leq w $. To see this, let $ G_R $ be the Dirichlet Green function of $-\Delta$ in $B_R$ and define
    \begin{align*}
        P_R(x)=\int_{B_R}G_R(x,y)h(y)\ \dy, \  \  x\in B_R.
    \end{align*}
    Since $-\Delta P_R=h$ in $B_R$ and $P_R=0$ on $\partial B_R$, we know that $ w - P_R $ is harmonic in $B_R$ and nonnegative on $\partial B_R$. Then, the maximum principle implies $ 0 \leq P_R(x) \leq w(x)$ in $B_R$.

    For fixed $x,y\in\mathbb R^N$, $x\neq y$, the Dirichlet Green functions $G_R(x,y)$ increase to $K(x-y)$ as $R\to\infty$. Hence, it follows from monotone convergence that
    \begin{align*}
        \mathcal{I}h (x)  = \int_{\mathbb R^N} K(x-y)h(y)\ \dy = \lim_{R\to\infty}P_R(x) \leq w(x)<\infty.
    \end{align*}
     Moreover, $w - \mathcal{I} h$ is a nonnegative harmonic function on $ \mathbb R^N $, which implies that $ w - \mathcal{I} h$ is a constant.

     Applying this representation to system \eqref{Main1}, we obtain
    \begin{align*}
        u = \mathcal{I} (v^p) + c_u, \ \ v = \mathcal{I} (u^q) + c_v,
    \end{align*}
    for some constants $ c_u, c_v \geq 0 $. We next prove \eqref{int-uv-bound}. Choose $ a \in \mathbb R^N $ and let
    \begin{align*}
        M_u(a) = \int_{B_1(a)} u^q(x) \ \dx, \ \ M_v(a) = \int_{B_1(a)} v^p(x) \ \dx.
    \end{align*}
    For $x,y\in B_1(a)$, the kernel $K(x-y)$ is bounded below by a positive constant $c_N$ depending only on $N$. Thus
    \begin{align}\label{M-u-M-v}
        u(x) \geq \mathcal{I}(v^p)(x) \geq \int_{B_1(a)} K(x-y) v^p(y) \ \dy \geq c_N M_v(a), \ \ v(x)\geq \mathcal{I} (u^q) (x) \geq c_N M_u(a),
    \end{align}
    for $ x \in B_1(a) $. Integrate the $p$ and $q$ power of \eqref{M-u-M-v} over $B_1(a)$, it follows that
    \begin{align}\label{Mu-Mv}
        M_u(a) \geq c M_v^q(a), \ \  M_v(a) \geq c M_u^p(a),
    \end{align}
    where $c$ is a constant depending on $N,p,q$. Combining the two inequalities in \eqref{Mu-Mv}, we obtain
    \begin{align*}
        M_u (a) \geq c M_u^{pq}(a), \ \  M_v(a) \geq c M_v^{pq}(a).
    \end{align*}
    From \eqref{critical-hyperbola}, we know that $ pq>1 $, and hence $M_u(a)$ and $M_v(a)$ are bounded above by a constant depending only on $N, p ,q$. Next, define
    \begin{align*}
        \tilde{u}(x) = r^\alpha u(a+rx), \ \ \tilde{v}(x)=r^\beta v(a+rx),
    \end{align*}
    where $\alpha=N/(q+1)$ and $\beta=N/(p+1)$. Then $( \tilde{u}, \tilde{v} )$ also solves \eqref{Main1}. Hence, we have
    \begin{align*}
        C & \geq \int_{B_1} \tilde{u}^q(x) +  \tilde{v}^p(x) \ \dx  \\
        & = r^{\alpha q - N} \int_{B_r(a)} u^q(x) \ \dx +  r^{\beta p - N} \int_{B_r(a)} v^p(x) \ \dx \\
        & = r^{ -\alpha } \int_{B_r(a)} u^q(x) \ \dx +  r^{ - \beta} \int_{B_r(a)} v^p(x) \ \dx .
    \end{align*}
    This proves \eqref{int-uv-bound}.

    We next prove that $c_u = c_v = 0$. Suppose on the contrary that $ c_u > 0 $. Since $u = \mathcal{I}(v^p) + c_u \geq c_u$, it follows from \eqref{int-uv-bound}  that
    \begin{align*}
        Cr^\alpha \geq \int_{B_r(a)} u^q(x) \ \dx \geq c_u^q |B_r| = |B_1(0)| c_u^q r^N.
    \end{align*}
    Because $ \alpha = N/(q+1) < N$, letting $ r \to \infty $ yields a contradiction. Therefore $c_u=0$. Similarly, using $v\geq c_v$ and $\beta=  N/(p+1)<N$, we obtain $c_v=0$. Consequently, \eqref{int-repre-uv} holds.

    Then, we prove \eqref{int-gradient-uv}. From the definition of $K(x)$, we have
    \begin{align*}
        |\nabla K(x)| = \frac{1}{|\mathbb{S}^{N-1}| \cdot |x|^{N-1}} =\frac{N-2}{|x|}K(x).
    \end{align*}
    Hence, for any fixed $x \in \R^N$, the integrals in \eqref{int-gradient-uv} are absolutely convergent, since
    \begin{align}\label{nabla-K}
        \int_{\R^N} |\nabla K(x-y)| v^p(y) \ \dy & =  \int_{|y-x| \leq 1} \left| \nabla K(x-y) \right| v^p(y) \ \dy + \int_{|y-x| \geq 1} \left| \nabla K(x-y) \right| v^p(y) \ \dy \nonumber \\
        & \leq C(N) ||v||_{L^{\infty}(B_1(x))}^p \int_{B_1} |y|^{1-N} \ \dy + (N-2) \int_{\R^N} K(x-y) v^p(y) \ \dy \nonumber \\
        & = C(N) ||v||_{L^{\infty}(B_1(x))}^p + (N-2) u(x) < +\infty.
    \end{align}
    To obtain \eqref{int-gradient-uv}, we split the integral on $\R^N$ into $ B_{R} $ and $B_R^c$ for some $R>0$. The bounded part can be differentiated since $\nabla K\in L^1_{\mathrm{loc}}(\mathbb R^N)$. The unbounded part can be differentiated too by using \eqref{nabla-K} and dominated convergence.

    Finally, if $u$ and $v$ are bounded, a similar argument as \eqref{nabla-K} yields
    \begin{align*}
       |\nabla u (x)| & = \left|  \int_{\R^N} \nabla K(x-y) v^p(y) \ \dy \right| \leq C(N) ||v||_{L^{\infty}}^p + (N-2) ||u||_{L^{\infty}}.
    \end{align*}
    The estimate of $\nabla v$ is similar. Therefore, \eqref{gradient-uv-bound} holds.
\end{proof}

By Proposition \ref{lem-int-repre}, system \eqref{Main1} is equivalent to its integral representation. We will prove that every positive solution $(u,v)$ is radial about a common point. The classification then follows from a comparison for the radial equations. We first prove radial uniqueness for \eqref{Main1}.

Since $U$ and $V$ are radial, we can write system \eqref{Main1} as $$- (r^{N-1} U')' = r^{N-1} V^p \hbox{ and }- (r^{N-1} V')' = r^{N-1} U^q.$$ It then follows from integration that
\begin{align}\label{rU'-rV'}
    - rU'(r)  = r^{-N+2} \int_0^r s^{N-1} V^p(s) \ \ds, \ \ - rV'(r)  = r^{-N+2} \int_0^r s^{N-1} U^q(s) \ \ds.
\end{align}
Integrating once more and using the decay at infinity, we obtain
\begin{align}
    U(r) & = \frac{1}{N-2} \left( \frac{1}{r^{N-2}} \int_0^r s^{N-1} V^p(s) \ \ds + \int_r^{\infty} s V^p(s) \ \ds \right), \label{Ur}\\
    V(r) & = \frac{1}{N-2} \left( \frac{1}{r^{N-2}} \int_0^r s^{N-1} U^q(s) \ \ds + \int_r^{\infty} s U^q (s) \ \ds \right). \label{Vr}
\end{align}
Hence, from $U,V>0$ and \eqref{rU'-rV'}-\eqref{Vr}, we obtain
\begin{align}\label{U'V'-UV}
    -r U'(r) \leq (N-2)U (r) , \ \ -rV'(r) \leq (N-2)V(r).
\end{align}
Moreover, since
\begin{align*}
    \int_0^r s^{N-1} U^q(s) \ \ds \leq (N-2)r^{N-2}V(r),
\end{align*}
and $ V(r) \leq Cr^{-N+2}$, we know that $\int_0^r s^{N-1} U^q(s) \ \ds$ is uniformly bounded with respect to $r$ and thus $ U^q \in L^1(\R^N)$. We now prove radial uniqueness for \eqref{Main1}.

\begin{lemma}\label{radial-uniqueness}
    Suppose $(u,v)$ is a positive solution of \eqref{Main1}, and $ u $, $ v $ are radial about a common point. Then $(u,v)$ is a translation and dilation of the ground state $(U,V)$.
\end{lemma}

\begin{proof}
    After a translation, we may assume that the common center is the origin. We first prove that $u$ and $v$ both decay to zero at infinity. Since $u$ and $v$ are radial, we have
    \begin{align*}
        u'(r) & = -r^{1-N} \int_0^r s^{N-1} v^p(s) \ \ds < 0,\\
        v'(r) & = -r^{1-N} \int_0^r s^{N-1} u^q(s) \ \ds < 0.
    \end{align*}
    Thus both $u$ and $v$ are strictly decreasing and positive. Taking $ a=0 $ in \eqref{int-uv-bound}, we obtain
    \begin{align*}
        u^q(r) \leq \frac{1}{|B_1| r^N}\int_{B_r(a)} u^q(x) \ \dx \leq C r^{\alpha - N}, \ \ v^p(r) \leq \frac{1}{|B_1| r^N}\int_{B_r(a)} v^p(x) \ \dx \leq C r^{\beta - N}.
    \end{align*}
    Since $\alpha < N, \beta < N$, we have $ v(r) \to 0$ and $ u(r) \to 0$ as $ r \to +\infty$.

    We next show that a positive radial solution is uniquely determined by the value of its first component at the origin. Let $(u_1,v_1)$ and $(u_2,v_2)$ be two positive radial entire solutions satisfying $ u_1(0)=u_2(0) $. We first prove $v_1 (0) = v_2(0)$. Suppose, for a contradiction, that $v_1(0) > v_2(0)$. Let
    \begin{align*}
        h=u_1-u_2,\ \  k=v_1-v_2.
    \end{align*}
    Then $h(0)=0$ and $k(0)>0$ and
    \begin{align}
        h'(r)&=-r^{1-N}\int_0^r s^{N-1}\bigl(v_1^p(s) -v_2^p (s) \bigr)\ \ds, \label{h'}\\
        k'(r)&=-r^{1-N}\int_0^r s^{N-1}\bigl(u_1^q(s) -u_2^q(s) \bigr)\ \ds. \label{k'}
    \end{align}
    By continuity, $ k > 0 $ near the origin and hence $ h'< 0 $, which implies $h(r) < h(0) = 0$ near the origin. By $h(r)<0$ and \eqref{k'}, we have $k'(r)>0$, so $k(r)>k(0)>0$ near the origin. These signs persist by continuity and \eqref{h'}--\eqref{k'}, so $h$ is decreasing and $k$ is increasing on $\R_+$. Thus
    \begin{align*}
        h(r) < 0, \ \  k(r) = v_1(r) - v_2(r) \geq k(0)>0
    \end{align*}
    for all $r>0$. However, this contradicts the fact that $v_1(r), v_2(r) \to 0$ as $r\to\infty$. The case $v_1(0)<v_2(0)$ is also excluded similarly. Hence $ v_1 (0) = v_2(0)$.

    Next, we prove that the two solutions are equal. Integrating the radial equations twice, we have
    \begin{align}
        u_i(r)&=u_i(0)-\int_0^r t^{1-N}\int_0^t s^{N-1} v_i^p (s) \ \ds \dt, \label{u-i}\\
        v_i(r)&=v_i(0)-\int_0^r t^{1-N}\int_0^t s^{N-1} u_i^q(s) \ \ds \dt, \label{v-i}
    \end{align}
    for $i=1,2$. By positivity and continuity, there exist $\rho_0>0$ and $0<m<M$ such that $m\leq u_i(r),v_i(r)\leq M$ for $0\leq r\leq\rho_0$. Let
    \begin{align*}
        L=\max_{m\leq t\leq M}p t^{p-1}+\max_{m\leq t\leq M}qt^{q-1}.
    \end{align*}
    For $ 0 < \rho \leq \rho_0 $, let
    \begin{align*}
        M_\rho = \sup_{0 \leq r \leq \rho }\bigl( |h(r)| + |k(r)| \bigr),
    \end{align*}
    then it follows from \eqref{u-i}, \eqref{v-i} and $ u_1(0) = u_2(0) $, $ v_1(0) = v_2(0) $ that
    \begin{align*}
        M_\rho & \leq \int_0^{\rho} t^{1 - N} \int_0^t s^{N-1} \big( |v_1^p (s) - v_2^p (s)| + | u_1^q (s) - u_2^q (s)| \big)  \ \ds \dt \\
        & \leq L M_{\rho} \int_{0}^{\rho} t^{1-N} \int_0^{t} s^{N-1} \ \ds \dt = \frac{L\rho^2}{2N} M_{\rho} .
    \end{align*}
    Choosing $\rho$ so that $L\rho^2/(2N)<1$, we obtain $M_\rho=0$. Standard ODE uniqueness on intervals away from $r=0$ then gives $(u_1,v_1)=(u_2,v_2)$.

    Finally, choose $\lambda>0$ such that $ \lambda^\alpha U(0)=u(0) $.  Then the critically rescaled ground state
    \begin{align*}
        U_\lambda (r) = \lambda^\alpha U( \lambda r ), \ \  V_\lambda (r) = \lambda^\beta V( \lambda r )
    \end{align*}
    is another positive radial solution with $ U_\lambda(0) = u(0)$. Hence, we have
    \begin{align*}
        u(r) = \lambda^\alpha U(\lambda r), \ \ v(r) = \lambda^\beta V(\lambda r).
    \end{align*}
\end{proof}

By the preceding lemma, it remains to prove radial symmetry for every positive solution. In this section, we prove the following isolation theorem for ground state $(U,V)$, which plays a fundamental role in the blow up analysis in the rest of the paper.

\begin{thm}\label{Isolate-thm}
    Suppose $(U,V)$ is the bounded radial ground state of \eqref{Main1} with $U(0)=1$. There exist constants $ \delta_0, R_0>0$, depending only on $N,p,q$, such that for every positive classical solution $(u,v)$ of \eqref{Main1}, if
    \begin{align*}
        u(0) =1, \ \ \nabla u(0) =0, \ \ || u - U||_{C^2(\overline{B_{R_0}})} + || v - V ||_{C^2(\overline{B_{R_0}})} < \delta_0,
    \end{align*}
    then $ (u,v) = (U,V) $.
\end{thm}

To prove this theorem, we need several lemmas that allow us to control the difference between a general positive solution and the ground state.

\begin{lemma}\label{mixed-integral}
    Suppose $(U,V)$ is the ground state of \eqref{Main1} and $(u,v)$ is a positive solution of \eqref{Main1}. Then
    \begin{align*}
        0 < \int_{\R^N} U^q u = \int_{\R^N} V v^p \leq E_* : = \int_{\R^N} U^{q+1} = \int_{\R^N} V^{p+1}.
    \end{align*}
    Equality in the inequality holds if and only if $ (u,v) = (U, V)$.
\end{lemma}

\begin{proof}
   We first show that the integrals in the statement are finite. By Theorem \ref{Decay-GS}, the ground state satisfies $V(x)\leq C(1+|x|)^{2-N}$. Thus, by Proposition \ref{lem-int-repre},
    \begin{align*}
        \int_{\R^N} V v^p & = \int_{B_1} V v^p + \sum_{j=0}^{\infty} \int_{\{2^j<|x|<2^{j+1}\}} V v^p  \leq C \left( \int_{B_1} v^p + \sum_{j=0}^{\infty} 2^{-j(N-2)} \int_{B_{2^{j+1}}} v^p  \right) \\
        & \leq C + C \sum_{j=0}^{\infty} 2^{(j+1) \beta - j(N-2)} \leq C + C \sum_{j=0}^{\infty} 2^{-j\alpha} < +\infty,
    \end{align*}
    where we have used the fact that $ \alpha = N/(q+1)$, $ \beta = N/(p+1) $ and $\alpha + \beta = N -2$. Let $S  =\int_{\R^N} V v^p $, then from Proposition \ref{lem-int-repre}, we have
    \begin{align*}
        S = \int_{\R^N} V v^p =\int_{\R^N} \mathcal{I}( U^q) v^p =  \int_{\R^N} U^q \mathcal{I}(v^p) =  \int_{\R^N}U^q u.
    \end{align*}

    We next prove $S \leq E_*$. Since $q>1$, $v=\mathcal I(u^q)$, and $V=\mathcal I(U^q)$, H\"older's inequality gives
    \begin{align*}
        \mathcal{I}(U^{q-1}u)(x) & \leq \bigl( \mathcal{I} (u^q)(x) \bigr)^{1/q} \bigl( \mathcal{I}(U^q)(x) \bigr)^{(q-1)/q} = v(x)^{1/q}V(x)^{(q-1)/q}.
    \end{align*}
    It follows that
    \begin{align}\label{S-1}
        S  = \int_{\R^N}V^{p+1} \left( \frac{v}{V} \right)^p \geq \int_{\R^N} V^{p+1} \left( \frac{ \mathcal{I}(U^{q-1}u)}{V} \right)^{pq}.
    \end{align}
    Since $pq>1$, H\"older's inequality implies
    \begin{align}\label{S-2}
        \int_{\R^N} V^p \mathcal{I}( U^{q-1} u) & =  \int_{\R^N} V^{p+1} \frac{\mathcal{I}( U^{q-1} u)}{V} \\
        & \leq \left( \int_{\R^N} V^{p+1} \right)^\frac{pq-1}{pq} \left( \int_{\R^N} V^{p+1}  \left( \frac{ \mathcal{I}(U^{q-1}u)}{V} \right)^{pq}\right)^\frac{1}{pq}. \nonumber
    \end{align}
    Hence, from \eqref{S-1} and \eqref{S-2}, we obtain
    \begin{align*}
        S \geq \left( \int_{\R^N} V^{p+1} \right)^{1-pq} \left( \int_{\R^N} V^p \mathcal{I} (U^{q-1}u) \right)^{pq} = E_*^{1-pq} \left( \int_{\R^N} U^{q-1} u \mathcal{I}(V^p) \right)^{pq} = E_*^{1-pq} S^{pq}.
    \end{align*}
    Since $S>0$ and $pq>1$, this implies $S \leq E_*$.

    Finally, suppose that $S=E_*$. Then equality must hold in both H\"older's inequalities. In particular, equality in the first one implies that, for almost every $x\in\R^N$, $u$ and $U$ are proportional almost everywhere with respect to $K(x-y)\, \dy$. Since $K$ is strictly positive, there exists a constant $t>0$ such that $u=tU$ in $\R^N$. Therefore, we have
    \begin{align*}
        v= \mathcal{I}(u^q)=t^q V,\ \  u= \mathcal{I}(v^p)=t^{pq}U.
    \end{align*}
    Therefore $t=t^{pq}$, and hence $t=1$. Thus $(u,v)=(U,V)$.
\end{proof}

In the proof of Theorem \ref{Isolate-thm}, we will divide the integral representation on $\R^N$ into $B_R$ and $ B_R^c $ for some $ R>0 $. For this reason, we need the following lemma to deal with the unbounded part.

\begin{lemma}\label{lem-kappa-R}
    Suppose $(U,V)$ is a ground state of \eqref{Main1}. Then
    \begin{align*}
        \kappa_R = \sup_{x \in \R^N} \frac{\mathcal{I}(U^q \chi_{ \{ |y| \geq R \} }) (x)}{V(x)} \to 0,
    \end{align*}
    as $R \to +\infty$.
\end{lemma}

\begin{proof}
    Write
    \begin{align*}
        m_R :=  \int_R^\infty s^{N-1} U(s)^q\ \ds.
    \end{align*}
    Since $U^q\in L^1(\R^N)$, we have $m_R \to 0$ as $R \to \infty$.

    We next estimate $\mathcal{I} (U^q \chi_{ \{ |y| \geq R \} }) $. Write $ r = |x|$. If $r\geq R$, the radial representation similar to \eqref{Ur} and \eqref{Vr} gives
    \begin{align}\label{tail-1}
        \mathcal{I} ( U^q \chi_{\{|y|>R\}})(r) & = \frac{1}{N-2} \left(r^{-N+2} \int_R^r s^{N-1} U(s)^q \ \ds + \int_r^\infty s U(s)^q \ \ds \right) \nonumber \\
        & \leq \frac{r^{-N+2}}{N-2} \left( \int_R^r s^{N-1} U(s)^q \ \ds + \int_r^\infty s^{N-1} U(s)^q \ \ds \right) \\
        &\leq Cr^{ -N + 2} m_R. \nonumber
    \end{align}
    If $0 \leq r < R$, the same formula gives
    \begin{align}\label{tail-2}
        \mathcal{I} ( U^q \chi_{\{|y|>R\}})(r)  = \frac{1}{N-2} \int_R^\infty s U(s)^q\ \ds \leq C R^{-N+2} m_R.
    \end{align}
    Hence, from \eqref{tail-1} and \eqref{tail-2}, we obtain
    \begin{align}\label{tail-3}
        \mathcal{I} ( U^q \chi_{\{|y|>R\}})(r) \leq Cm_R\max\{r,R\}^{-N +2}.
    \end{align}
    From Theorem \ref{Decay-GS}, we know that $V(r)\geq c(1+r)^{2-N}$ for some $c>0$. Since $ R \geq 1$, we obtain from \eqref{tail-3} that
    \begin{align*}
        \frac{\mathcal{I} ( U^q \chi_{\{|y|>R\}} )(r)}{V(r)} \leq C m_R \left(\frac{1+r}{\max\{r,R\}} \right)^{N-2} \leq C m_R,
    \end{align*}
    for some constant $C>0$, which does not depend on $r$. Consequently, $ \kappa_R \leq Cm_R \to 0$, as $ R \to +\infty$.
\end{proof}

We use this lemma to estimate solutions of the linearized system.

\begin{lemma}\label{linear-estimate}
    Suppose $(U,V)$ is a ground state of \eqref{Main1}, and $ (\phi, \psi) \in C^2(\R^N) \times C^2(\R^N) $ satisfies
    \begin{align}\label{linearized-equ}
        \begin{cases}
            - \Delta \phi = p V^{p-1} \psi, \ \ & \text{in} \ \ \R^N, \\
            - \Delta \psi = q U^{q-1} \phi, \ \ & \text{in} \ \ \R^N.
        \end{cases}
    \end{align}
    Assume also that there is a constant $C_0>0$ such that
    \begin{align*}
        \phi \geq - C_0 U, \ \ \psi \geq -C_0 V, \ \ \text{in} \ \ \R^N, \ \ \int_{\R^N} ( U^q |\phi| + V^p |\psi| ) < +\infty.
    \end{align*}
    Then there is a constant $C>0$ such that
    \begin{align*}
        |\phi| \leq CU, \ \ |\psi| \leq C V.
    \end{align*}
\end{lemma}

\begin{proof}
    We first write the linearized system in integral form. Choose $ M>0 $ sufficiently large so that $M \geq C_0\max\{1,p,q\}$. Then
    \begin{align*}
        \phi + M U \geq 0, \ \  -\Delta( \phi + MU) = p V^{p-1} \psi + M V^p \geq 0,
    \end{align*}
    and similarly,
    \begin{align*}
        \psi + M V \geq 0,\ \ -\Delta( \psi + MV ) = q U^{q-1} \phi + M U^q \geq 0.
    \end{align*}
    As in the proof of Proposition \ref{lem-int-repre}, there are constants $c,h\geq0$ such that
    \begin{align*}
        \phi + M U=p \mathcal{I} (V^{p-1} \psi ) + M \mathcal{I} ( V^p)+c,\ \  \psi + M V = q \mathcal{I} (U^{q-1} \phi ) + M \mathcal{I} ( U^q ) + h.
    \end{align*}
    Since $ U = \mathcal{I}(V^p), V = \mathcal{I}(U^q)$, we have
    \begin{align}\label{integral-ch}
        \phi =p \mathcal{I} (V^{p-1} \psi ) + c,\ \  \psi  = q \mathcal{I} (U^{q-1} \phi )  + h.
    \end{align}
    
    We now prove that $c=h=0$. Let $(U_{\lambda}(x), V_{\lambda}(x)) = (\lambda^{\alpha}U(\lambda x), \lambda^{\beta} V(\lambda x))$, then it follows from Proposition \ref{lem-int-repre} that $ U_{\lambda} = \mathcal{I} (V_{\lambda}^p)$ and $ V_{\lambda} = \mathcal{I} (U_{\lambda}^q)$.     Differentiating the integral system at $\lambda=1$ gives
    \begin{align*}
       Z_U=\mathcal I(Q_V), \ \  Z_V=\mathcal I(Q_U),
    \end{align*}
    where
    \begin{align*}
        Z_U = \alpha U+x\cdot\nabla U,\ \ Z_V=\beta V + x \cdot \nabla V, \ \ Q_U = q U^{q-1} Z_U,\ \ Q_V = p V^{p-1} Z_V.
    \end{align*}
    Moreover, \eqref{U'V'-UV} implies $|Z_U| \leq CU$, $|Z_V|\leq CV$, and hence
    \begin{align*}
        |Q_U| \leq C U^q, \ \ |Q_V| \leq C V^p.
    \end{align*}
    The following changes in the order of integration are justified by
    \begin{align*}
        \int_{\R^N}|Q_U|\mathcal I(V^{p-1}|\psi|)
        \leq C\int_{\R^N}U^q\mathcal I(V^{p-1}|\psi|)
        =C\int_{\R^N}V^p|\psi|<\infty,
    \end{align*}
    and the corresponding estimate with the components interchanged.

    Next, multiplying the first equation in \eqref{integral-ch} by $Q_U$ and integrating, we obtain
    \begin{align*}
        \int_{\R^N} Q_U \phi & = p \int_{\R^N} Q_U \mathcal{I} (V^{p-1} \psi) + c \int_{\R^N} Q_U\\
        & = p \int_{\R^N} V^{p-1} \psi \mathcal{I} (Q_U) + c \int_{\R^N} Q_U\\
        & = \int_{\R^N} Q_V \psi + c\int_{\R^N} Q_U.
    \end{align*}
    Here $Q_U$ is integrable, because $ U^q \in L^1(\R^N)$.

    Suppose first that $V^p\in L^1(\R^N)$. Then $Q_V$ is also integrable, and the equation for $\psi$ similarly gives
    \begin{align*}
        \int_{\R^N} Q_V \psi = \int_{\R^N} Q_U \phi+ h \int_{\R^N} Q_V.
    \end{align*}
    Combining the two identities gives
    \begin{align*}
        c\int_{\R^N}Q_U+h\int_{\R^N}Q_V=0.
    \end{align*}
    Since $U^q$ and $V^p$ are positive radial decreasing integrable functions, integration by parts gives
    \begin{align*}
        \int_{\R^N} Q_U & = (q\alpha - N)\int_{\R^N} U^q = - \alpha \int_{\R^N} U^q<0,\\
        \int_{\R^N} Q_V & = (p\beta - N) \int_{\R^N} V^p =- \beta \int_{\R^N} V^p<0.
    \end{align*}
    The boundary terms vanish by radial monotonicity and integrability. Since $c,h\geq0$, we conclude that $c=h=0$.

    If instead $V^p\notin L^1(\R^N)$, we first prove that $ h=0 $. Indeed, the lower bound $\phi \geq- C_0 U$ implies
    \begin{align*}
        \psi = q \mathcal{I} (U^{q-1}\phi) + h \geq -qC_0 \mathcal{I} (U^q) + h = h - q C_0 V.
    \end{align*}
    Since $V(x)\to0$ as $|x|\to\infty$, the assumption $h>0$ would imply $\psi\geq h/2$ near infinity, which contradicts $\int_{\R^N}V^p|\psi|<+\infty$. Hence $h=0$. Now, multiplying the second equation of \eqref{integral-ch} by $Q_V$, we obtain
    \begin{align*}
        c \int_{\R^N} Q_U=0,
    \end{align*}
    and therefore $c=0$ as well. Thus
    \begin{align}\label{integral-no-ch}
        \phi=p\mathcal I(V^{p-1}\psi),\ \ \psi=q\mathcal I(U^{q-1}\phi).
    \end{align}

    Next, we prove $|\phi| \leq CU$. Let $r=\phi/U$. Substituting the second integral equation in \eqref{integral-no-ch} into the first gives
    \begin{align*}
        r=pq\,\mathcal Pr,\ \  \mathcal Pr=\frac{1}{U} \mathcal{I} \left(V^{p-1} \mathcal{I}(U^q r)\right).
    \end{align*}
    Since the operator $\mathcal{I}$ is positive and $U,V>0$, the operator $\mathcal P$ is also positive and satisfies
    \begin{align*}
        \mathcal P1=\frac{1}{U} \mathcal{I} \left(V^{p-1} \mathcal{I}(U^q)\right) =\frac{ \mathcal{I}(V^p)}{U}=1.
    \end{align*}
    Moreover, for nonnegative functions $a,b$, it holds that
    \begin{align*}
        \int_{\R^N} U^{q+1}a \,\mathcal P b =\int_{\R^N} V^{p-1} \mathcal{I}(U^qa) \mathcal{I}(U^qb) = \int_{\R^N} U^{q+1} b \,\mathcal Pa.
    \end{align*}
    Consequently, using $\mathcal P1=1$ and the positivity of $\mathcal{P}$, we obtain
    \begin{align*}
        \| \mathcal Pr \|_{L^\infty} \leq\|r \|_{L^\infty}, \ \ \int_{\R^N}U^{q+1}|\mathcal Pr| \leq \int_{\R^N}U^{q+1}|r|.
    \end{align*}
    In particular, $\mathcal P$ is bounded on both $L^\infty(\R^N)$ and $L^1(\R^N, U^{q+1}  \ \dx)$, with operator norm at most $1$.

    Next, by Lemma \ref{lem-kappa-R}, we have
    \begin{align*}
        \mathcal I(U^q \chi_{B_R^c}) \leq \kappa_R V, \  \ \kappa_R \to 0 \ \ \text{as} \ \ R \to \infty.
    \end{align*}
    It follows that
    \begin{align*}
        \mathcal P\chi_{B_R^c} =\frac{1}{U} \mathcal I \left( V^{p-1} \mathcal I( U^q \chi_{B_R^c})\right) \leq \frac{\kappa_R}{U} \mathcal I(V^p) = \kappa_R.
    \end{align*}
    Define
    \begin{align*}
        T_R r = \chi_{B_R^c}\mathcal P(\chi_{B_R^c}r).
    \end{align*}
    Then, the above estimate implies
    \begin{align}\label{T-R-norm}
        \|T_R r\|_{L^\infty} = ||\chi_{B_R^c}\mathcal P(\chi_{B_R^c}r)||_{L^{\infty}} \leq ||r||_{L^{\infty}} ||\mathcal P \chi_{B_R^c}||_{L^\infty} \leq \kappa_R ||r||_{L^{\infty}}.
    \end{align}
    Moreover, we have
    \begin{align}\label{T-R-norm2}
        \int_{\R^N}U^{q+1} |T_R r|  \leq \int_{B_R^c} U^{q+1} \mathcal P(\chi_{B_R^c} |r|) = \int_{B_R^c} U^{q+1} |r|\,\mathcal P\chi_{B_R^c} \leq \kappa_R \int_{B_R^c} U^{q+1}|r|.
    \end{align}
    Thus $T_R$ has operator norm at most $\kappa_R$ on $L^{\infty}(\R^N)$ and $L^{1}(\R^N; U^{q+1} \ \dx)$.

    Choose $ R $ to be a large constant. Since $r=\phi/U$ is continuous and $U>0$, we have $|r|\leq M_R$ for some constant $M_R > 0$ on $B_R$. Next, we need to prove that $r \chi_{B_R^c}$ is also bounded. Restricting the equation $r=pq\mathcal Pr$ to $B_R^c$, we obtain
    \begin{align}\label{r-chi-DR}
        r\chi_{B_R^c} = p q T_R(r \chi_{B_R^c}) + F_R, \ \  F_R = pq \chi_{B_R^c} \mathcal P(r \chi_{B_R}).
    \end{align}
    From the positivity of $\mathcal P$, and $r \chi_{B_R} \in L^\infty $, we have
    \begin{align}
        | F_R | \leq pq M_R \mathcal{P}1 = pq M_R,
    \end{align}
    and thus $ F_R \in L^1(\R^N,U^{q+1} \dx)$ since $\int_{\R^N}U^{q+1}<+\infty$.

    Choose $ R $ large enough, such that $ p q \kappa_R < 1 $, and define
    \begin{align*}
        s_R = \sum_{k=0}^\infty(pq T_R)^k F_R.
    \end{align*}
    Then, \eqref{T-R-norm} implies that $s_R$ is well-defined in $L^\infty$ and is a bounded solution of equation \eqref{r-chi-DR}. Indeed, \eqref{T-R-norm} implies that
    \begin{align*}
        ||s_R||_{L^\infty} \leq \sum_{k=0}^{\infty} (pq\kappa_R)^k ||F_R||_{L^{\infty}} \leq \frac{pq M_R}{1-pq \kappa_R}.
    \end{align*}
    Next, we prove that $s_R = r \chi_{B_R^c}$ almost everywhere in $\R^N$. Both $s_R$ and $r\chi_{B_R^c}$ belong to $L^1(\R^N;U^{q+1}\dx)$, the latter by $\int U^q|\phi|<\infty$. Since they satisfy \eqref{r-chi-DR}, we have
    \begin{align*}
        (r \chi_{B_R^c} - s_R) = pq T_R( r \chi_{B_R^c} - s_R ).
    \end{align*}
    Hence, \eqref{T-R-norm2} gives
    \begin{align*}
        || r  \chi_{B_R^c} - s_R ||_{L^1( \R^N; U^{q+1} \dx)} \leq pq \kappa_R || r  \chi_{B_R^c} - s_R ||_{L^1( \R^N; U^{q+1} \dx)},
    \end{align*}
    which together with $pq \kappa_R <1$ implies that $ r\chi_{B_R^c} = s_R$ almost everywhere in $\R^N$. Hence, we have $||r||_{L^{\infty}} < \infty$. Since $r$ is continuous, we have $|\phi|\leq CU$. Finally, the integral equation for $\psi$ gives
    \begin{align*}
        |\psi| \leq q \mathcal{I}(U^{q-1}|\phi|) \leq C \mathcal{I} (U^q) = CV.
    \end{align*}
    The proof is complete.
\end{proof}

We now prove Theorem~\ref{Isolate-thm}.

\begin{proof}[Proof of Theorem \ref{Isolate-thm}]
    We argue by contradiction. If the conclusion were false, then there exists a sequence of positive solutions $(u_j,v_j)$ of \eqref{Main1}, different from $(U,V)$, such that
    \begin{align*}
        u_j(0)=1, \ \  \nabla u_j(0)=0, \ \ (u_j,v_j) \to (U,V) \ \ \text{in} \  \ C^2_{\mathrm{loc}}(\R^N).
    \end{align*}

    We first use the local convergence to get a uniform lower bound on the whole space. For each fixed $R>0$, the local convergence implies
    \begin{align*}
        u_j \geq (1-o_j(1))U, \ \ \text{in} \ \ B_R.
    \end{align*}
    Using $v_j=\mathcal I(u_j^q)$, we get
    \begin{align*}
        v_j \geq \mathcal{I} ( u_j^q \chi_{B_R} ) \geq (1-o_j(1))^q \mathcal{I} (U^q \chi_{B_R})  \geq (1-o_j(1))^q(1-\kappa_R)V, \text{ in } \R^N, 
    \end{align*}
   where $ \kappa_R \to 0 $ as $ R \to \infty $. Using $u_j=\mathcal I(v_j^p)$, we get a global lower bound for $u_j$. Letting first $ j \to \infty $ and then $ R \to \infty $, we conclude that
    \begin{align*}
        \varepsilon_j := \max \left\{ \left \| \frac{(U-u_j)_+}{U} \right\|_{L^\infty(\R^N)}, \left \| \frac{ ( V-v_j )_+}{V} \right \|_{ L^\infty (\R^N) } \right \} \to 0.
    \end{align*}
    Moreover, $ \varepsilon_j > 0 $. Indeed, if $ \varepsilon_j = 0$, then $ u_j \geq U $ and $ v_j \geq V $, and Lemma~\ref{mixed-integral} implies $ (u_j,v_j) = (U,V) $, which  contradicts to our assumption. Next, define
    \begin{align*}
        \phi_j = \frac{u_j-U}{\varepsilon_j}, \  \  \psi_j = \frac{v_j-V}{\varepsilon_j}.
    \end{align*}
    By the definition of $\varepsilon_j$, we obtain
    \begin{align*}
        \phi_j \geq - U, \ \ \psi_j \geq - V \ \ \text{in } \R^N.
    \end{align*}
    Moreover, we have
    \begin{align}\label{phi-j-psi-j}
        -\Delta \phi_j = a_j(x)\psi_j, \ \  -\Delta \psi_j = b_j(x) \phi_j,
    \end{align}
    where
    \begin{align*}
        a_j(x) & = p \int_0^1 \bigl( V(x) + t(v_j(x)-V(x)) \bigr)^{p-1}\ \dt,\\
        b_j(x) & = q \int_0^1 \bigl( U(x) + t(u_j(x)-U(x)) \bigr)^{q-1} \ \dt.
    \end{align*}
    Since $(u_j,v_j)\to(U,V)$ in $C^2_{\mathrm{loc}}(\mathbb{R}^N)$ and $U,V>0$, we have
    \begin{align*}
        a_j \to p V^{p-1},\ \  b_j\to q U^{q-1}
    \end{align*}
    locally uniformly. Moreover, on every fixed ball, these coefficients are uniformly positive and bounded, with uniform local H\"{o}lder bounds.

    We first establish uniform weighted integral estimates for $(\phi_j,\psi_j)$. From Lemma~\ref{mixed-integral}, we have
    \begin{align*}
        \int_{\R^N} U^q (u_j-U) = \int_{\R^N} U^q (u_j-U)_+ - \int_{\R^N} U^q (u_j-U)_- \leq 0.
    \end{align*}
    Since $(U-u_j)_+\leq\varepsilon_jU$, it follows that $(u_j - U)_- \leq \varepsilon_j U$, and hence
    \begin{align}\label{Uq-phi-j}
        \int_{\R^N} U^q |\phi_j| = \frac{1}{\varepsilon_j} \int_{\R^N} U^q \left( (u_j - U)_+ + (u_j - U)_- \right)  \leq \frac{2}{\varepsilon_j} \int_{\R^N} U^q(u_j - U)_-  \leq2E_*.
    \end{align}
    Similarly, from Lemma \ref{mixed-integral}, we get
    \begin{align*}
        \int_{\R^N} V ( v_j^p - V^p ) = \int_{\R^N} V (v_j^p-V^p)_+ - \int_{\R^N} V (v_j^p-V^p)_- \leq 0.
    \end{align*}
    For sufficiently large $ j $, we have $ \varepsilon_j \leq 1/2 $, and $v_j \geq (1-\varepsilon_j) V$ implies
    \begin{align*}
        ( v_j^p - V^p )_- = ( V^p - v_j^p )_+ \leq \bigl( 1-(1-\varepsilon_j)^p \bigr) V^p \leq C_p \varepsilon_j V^p.
    \end{align*}
    Consequently,
    \begin{align}\label{V-vj-p}
        \int_{\R^N} V\left| \frac{v_j^p-V^p}{\varepsilon_j} \right| \leq \frac{2}{\varepsilon_j} \int_{\R^N} V (v_j^p - V^p)_- \leq 2 C_p E_*.
    \end{align}
   From \eqref{Uq-phi-j}, \eqref{V-vj-p}, and the local positivity of $U,V$ and the uniform lower bound for $a_j$, we obtain that $\phi_j$ and $\psi_j$ are uniformly bounded in $L^1_{\mathrm{loc}}(\mathbb{R}^N)$. Using integral representation of \eqref{phi-j-psi-j} and applying Young inequality, we obtain uniform local $L_{\loc}^{s}$ bounds of $(\phi_j, \psi_j)$ for $ 1 < s < \frac{N}{N-2}$. The elliptic regularity and Sobolev embedding then yield uniform local $L^\infty$ bounds. Finally, the uniform H$\ddot{o}$lder bounds for the coefficients allow us to apply Schauder estimates. Passing to a subsequence, we obtain $ (\phi_j,\psi_j) \to (\phi,\psi) $ in $ C^2_{\mathrm{loc}} (\mathbb{R}^N) $, and thus
    \begin{align*}
        -\Delta \phi = p V^{p-1} \psi, \ \  -\Delta\psi = qU^{q-1} \phi.
    \end{align*}
    From the bounds for $\phi_j,\psi_j$, we get $\phi\geq-U$ and $\psi\geq-V$. Moreover, by Fatou's lemma, we have
    \begin{align*}
        \int_{\mathbb{R}^N}U^q|\phi|\ \dx<\infty, \ \  p\int_{\R^N} V^p |\psi| \leq \liminf_{j\to\infty} \int_{\R^N} V \left| \frac{v_j^p-V^p}{\varepsilon_j} \right| < + \infty.
    \end{align*}
    Thus $ (\phi, \psi) $ satisfies the assumptions of Lemma \ref{linear-estimate}.

    Before applying Lemma \ref{linear-estimate}, we show that $(\phi, \psi) \neq (0,0) $. Let
    \begin{align*}
        h_j=(U-u_j)_+,\ \  z_j=(V-v_j)_+.
    \end{align*}
    Since $q>1$ and $K$ is positive, we obtain
    \begin{align*}
        V(x) - v_j(x) & = \int_{\R^N} K(x-y) (U^q(y) - u_j^q(y)) \ \dy \leq \int_{\R^N} K(x-y) (U^q(y)-u_j^q(y))_+\ \dy  \\
        & \leq q \int_{\R^N} K(x-y) U^{q-1} (U- u_j)_+ \ \dy = q \mathcal{I}(U^{q-1} h_j) .
    \end{align*}
    Similarly, for $U-u_j$, we use the mean value theorem and $v_j\geq V/2$ for large $j$ when $p<1$ to get
    \begin{align*}
        z_j \leq q \mathcal{I} (U^{q-1}h_j),\ \ h_j\leq C_p \mathcal{I}(V^{p-1}z_j).
    \end{align*}

    Suppose that $(\phi,\psi)=(0,0)$. Then $ \phi_j \to 0$ and hence the local positivity of $U$ implies $ h_j/(\varepsilon_j U) \to 0$ as $ j \to \infty$ uniformly on every fixed ball, while in the whole space $\R^N$, we have $ 0 \leq h_j /\varepsilon_jU \leq 1 $. For fixed $R>0$, splitting $\R^N$ into $B_R$ and $ B_R^c $ gives
    \begin{align*}
        \frac{z_j(x)}{\varepsilon_j V(x)}  \leq q \frac{ \mathcal{I} (U^{q-1}h_j \chi_{B_R})(x)}{\varepsilon_j V(x)} + q\frac{ \mathcal{I} ( U^{q-1}h_j \chi_{ \R^N\setminus B_R} )(x)}{\varepsilon_jV(x)}  \leq q\left\| \frac{h_j}{\varepsilon_jU} \right\|_{L^\infty(B_R)} + q\kappa_R.
    \end{align*}
    Taking the supremum in $x$, then letting $j\to\infty$ and $R\to\infty$, we get
    \begin{align}\label{zj-V->0}
        \left\| \frac{z_j}{\varepsilon_j V} \right \|_{L^\infty(\R^N)} \to 0, \ \ \text{as} \ \ j \to +\infty.
    \end{align}
    Similarly, we have
    \begin{align*}
        \frac{h_j}{\varepsilon_j U}  \leq \frac{ C_p \mathcal{I}(V^{p-1} z_j) }{ \varepsilon_j U} \leq C_p \left\| \frac{z_j}{\varepsilon_j V} \right\|_{L^\infty (\R^N)}  \frac{\mathcal I(V^p)}{U}  = C_p\left\| \frac{z_j}{\varepsilon_j V} \right\|_{L^\infty(\R^N)}.
    \end{align*}
    Hence
    \begin{align}\label{hj-U->0}
        \left\| \frac{h_j}{\varepsilon_j U} \right \|_{L^\infty(\R^N)} \to 0.
    \end{align}
    However, the definition of $\varepsilon_j$ implies that
    \begin{align}\label{epsilon-j}
        \max\left\{ \left\|\frac{h_j}{\varepsilon_j U}\right\|_{L^\infty(\R^N)},\left\|\frac{z_j}{\varepsilon_j V}\right\|_{L^\infty(\R^N)}\right\}=1,
    \end{align}
    which contradicts \eqref{zj-V->0} and \eqref{hj-U->0}. Hence, it follows that $(\phi,\psi) \neq(0,0)$.

    Now, we apply Lemma \ref{linear-estimate} to $ (\phi, \psi) $ and get
    \begin{align*}
        |\phi| \leq C U,\ \ |\psi| \leq C V.
    \end{align*}
    In particular, it follows that $(\phi, \psi) \to 0$ as $|x| \to \infty$. Theorem \ref{non-degeneracy} then implies that
    \begin{align*}
        (\phi,\psi) = c_0 (Z_U,Z_V)+\sum_{i=1}^N c_i( \partial_iU, \partial_iV)
    \end{align*}
    for some constants $c_0,c_1,\ldots,c_N$. Since $ u_j(0) =1, \nabla u_j (0) = 0 $, we get
    \begin{align*}
        \phi(0)=0, \ \ \nabla\phi(0)=0.
    \end{align*}
    Since $U$ is radial and $U(0)=1$, we have
    \begin{align*}
        Z_U(0)=\alpha>0,\ \ \nabla U(0)=0.
    \end{align*}
    Thus $ \phi(0)=0 $ implies $ c_0=0 $. Moreover, since $U$ is radial and satisfies $ - \Delta U = V^p$, we get
    \begin{align*}
        D^2U (0)= -\frac{V(0)^p}{N} \operatorname{Id},
    \end{align*}
    which is invertible. Hence
    \begin{align*}
        0=\nabla\phi(0)=D^2 U(0) (c_1,\ldots,c_N)^T
    \end{align*}
    implies $c_1=\cdots=c_N=0$. Consequently $(\phi,\psi)=(0,0)$, a contradiction. This completes the proof.
\end{proof}

\section{A nonnegative defect and radial symmetry}\label{sec:defect}

In this section, we define a nonnegative quantity and show that its vanishing implies that $u$ and $v$ are radially symmetric about a common point.

Throughout this section, we write
\begin{align*}
    f=u^q, \ \  g=v^p,\ \ H = fu+gv = u^{q+1}+v^{p+1},
\end{align*}
and recall that
\begin{align*}
    \alpha=\frac{N}{q+1},\ \  d = N-2, \ \  \beta=\frac{N}{p+1}, \ \ \alpha+\beta = d .
\end{align*}
From Proposition \ref{lem-int-repre}, we have $u= \mathcal{I} g $ and $ v= \mathcal{I} f $, and
\begin{align*}
    \int_{B_r(a)} f \leq C r^\alpha, \ \  \int_{B_r(a)} g \leq Cr^\beta,
\end{align*}
which hold for every $ a \in \R^N $ and $ r > 0 $.

Fix $e \in\mathbb S^{N-1}$ and write
\begin{align*}
    x=z+se, \ \ y=w+s'e,\ \  z,w\in e^\perp.
\end{align*}
For $z\neq w$, set $\rho=|z-w|>0$ and define
\begin{align*}
    k(s,s')=K(z-w+(s-s')e)=c_N\bigl(\rho^2+(s-s')^2\bigr)^{-d/2}.
\end{align*}
For two bounded nonempty intervals $I=(s_-,s_+)$ and $J=(s'_-,s'_+)$ on these parallel lines, denote their midpoints by $m_I$ and $m_J$, respectively, and define
\begin{align*}
    T_e(I,J)=\int_{I\times J}\partial_s k(s,s')\ \ds\ds', \ \  \Delta_e(I,J)=-(m_I-m_J)T_e(I,J).
\end{align*}
Set
\begin{align*}
    k_I(s') = \frac{k(s_-,s') + k(s_+,s')}{2}, \ \  k_J(s) = \frac{ k(s,s'_-) + k(s,s'_+)}{2}.
\end{align*}

We have the following lemma.

\begin{lemma}\label{lem:interval-force}
    Let $I$ and $J$ be bounded nonempty intervals on distinct parallel lines. Then
    \begin{align}\label{interval-force-1}
        \Delta_e(I,J) \geq 0,
    \end{align}
    and $ \Delta_e(I,J)=0 $ if and only if $m_I = m_J$. Moreover,
    \begin{align}\label{interval-force-2}
        \int_{I\times J} (2k-k_I-k_J) \ \ds\ds' + \Delta_e(I,J) = d \int_{I\times J} \frac{(s-s')^2}{\rho^2 + (s-s')^2} k(s,s') \ \ds \ds'.
    \end{align}
    In particular, for any nonnegative constants $\gamma_I$ and $\gamma_J$,
    \begin{align}\label{interval-force-3}
        &\int_{I\times J}(\gamma_I + \gamma_J)k\ \ds  \ds' + \min\{ \gamma_I, \gamma_J \} \Delta_e(I,J)  \\
        &\ \  \leq \int_{I\times J}\left[ \gamma_I k_I + \gamma_J k_J + \left( d \min\{ \gamma_I, \gamma_J \} \frac{(s-s')^2}{\rho^2+ (s-s')^2} + |\gamma_I - \gamma_J| \right) k \right] \ \ds \ds'. \nonumber
    \end{align}
\end{lemma}

\begin{proof}
    We first determine the sign of $T_e(I,J)$. Write $I=(m_I-a,m_I+a)$ and $J=(m_J-b,m_J+b)$, where $a,b>0$, and set
    \begin{align*}
        k_\rho(t)=c_N(\rho^2+t^2)^{-d/2}, \ \ Q(t)=\int_{-b}^b k_\rho(t-\tau)\, \text{d} \tau.
    \end{align*}
    Since $k_\rho$ is even and strictly decreasing on $(0,\infty)$, the function $Q$ is even and satisfies
    \begin{align*}
        Q'(t)=k_\rho(t+b)-k_\rho(t-b)<0,\ \  t>0.
    \end{align*}
    Integrating $\partial_s k$ over $I$, we obtain
    \begin{align*}
        T_e(I,J) = Q(m_I-m_J+a) - Q(m_I-m_J-a).
    \end{align*}
    It follows that $T_e(I,J)$ has the opposite sign to $m_I-m_J$, and vanishes precisely when $m_I=m_J$. This proves \eqref{interval-force-1}.

    We next derive \eqref{interval-force-2}. Using integration by parts in $s$, we obtain
    \begin{align}
        \int_{I\times J} (k-k_I) \ \ds \ds' = - \int_{I\times J} (s-m_I) \partial_s k \ \ds \ds', \label{interval-force-4} \\
        \int_{I\times J}(k-k_J)\ \ds \ds' = - \int_{I\times J}(s'-m_J) \partial_{s'}k\ \ds \ds'. \label{interval-force-5}
    \end{align}
    Since
    \begin{align*}
        \partial_s k = - d \frac{s-s'}{\rho^2+(s-s')^2}k, \ \  \partial_{s'} k = - \partial_s k,
    \end{align*}
    we obtain by adding \eqref{interval-force-4} and \eqref{interval-force-5} that
    \begin{align}\label{interval-force-6}
        \int_{I\times J}(2k-k_I-k_J)\ \ds \ds' = d \int_{I\times J} \frac{(s-s')^2-(m_I-m_J)(s-s')}{\rho^2+(s-s')^2}k \ \ds \ds'.
    \end{align}
    On the other hand, we have
    \begin{align*}
        \Delta_e (I,J) = d(m_I-m_J) \int_{I\times J}\frac{s-s'}{\rho^2+(s-s')^2} k\ \ds \ds',
    \end{align*}
    which combined with \eqref{interval-force-6} yields \eqref{interval-force-2}.

    Finally, let $\gamma=\min\{\gamma_I,\gamma_J\}$. Multiplying \eqref{interval-force-2} by $\gamma$, we obtain
    \begin{align}\label{interval-force-7}
        \int_{I\times J}\gamma(2k-k_I-k_J)\ \ds \ds' + \gamma\Delta_e(I,J) = d \gamma \int_{I \times J } \frac{(s-s')^2}{\rho^2+(s-s')^2} k\ \ds \ds'.
    \end{align}
    Since $ k_I,k_J \geq 0 $, we also have
    \begin{align}\label{interval-force-8}
        \int_{I\times J}(k-k_I)\ \ds \ds' \leq \int_{I\times J} k\ \ds \ds', \ \ \int_{I\times J}(k-k_J)\ \ds \ds' \leq \int_{I\times J} k \ds \ds'.
    \end{align}
    Multiplying the first inequality of \eqref{interval-force-8} by $\gamma_I-\gamma$ and the second by $\gamma_J-\gamma$ respectively, and adding them to \eqref{interval-force-7} proves \eqref{interval-force-3}, because
    \begin{align*}
        (\gamma_I-\gamma)+(\gamma_J-\gamma)=|\gamma_I-\gamma_J|.
    \end{align*}
    This completes the proof.
\end{proof}

Recall that $f=u^q$ and $g=v^p$. For a fixed direction $e\in\mathbb S^{N-1}$, any point $ x \in \R^N$ can be written as $x=z+se$, where $z\in e^\perp$ and $s\in\R$. Given $z\in e^\perp$ and $t>0$, let $\mathcal I_f(e,z,t)$ denote the family of connected components of the open set
\begin{align*}
    \{s\in\R:f(z+se)>t\}.
\end{align*}
Each component is an open interval, which may be unbounded.

Since the components of $ \mathcal{I}_{f}(e,z,t)$ are pairwise disjoint, we have
\begin{align*}
    f(z+se) = \int_0^{\infty} \chi_{\{ f(z + se) > t\}} \ \dt =\int_0^\infty \sum_{I\in\mathcal I_f(e,z,t)}\chi_I(s)\ \dt.
\end{align*}
Similarly, define $\mathcal I_g(e,w,\tau)$ for $g$ and $ e \in \mathbb{S}^{N-1}, w \in e^{\perp}$. Then
\begin{align}\label{repre-fgK}
    & \ \  \int_{\R^N} \int_{\R^N} f(x) g(y) K(x-y)\ \dx \dy \nonumber \\
    & =  \int_{e^\perp} \int_{e^\perp} \int_{\R} \int_{\R} f(z+s e) g(w + s' e) K( z - w + (s-s') e )  \ \text{d}s \ds' \text{d}w \text{d}z \\
    & = \int_{e^\perp} \int_{e^\perp} \int_0^\infty \int_0^\infty \sum_{I \in \mathcal I_f(e,z,t)} \sum_{J \in \mathcal I_g(e,w,\tau)} \left( \int_{ I \times J} k(s,s') \ \ds \ds' \right) \ \text{d}\tau \dt \text{d}w \text{d}z. \nonumber
\end{align}
Here, we recall that $ k(s,s')= K(z-w+(s-s')e) $.

This representation allows us to apply Lemma \ref{lem:interval-force} to the solution $(u,v)$. In particular, if $u$ and $v$ are radial and decreasing about a common point $a \in \R^N$, then for every $e \in \mathbb{S}^{N-1} $, every bounded nonempty superlevel interval of $f$ and $g$ has midpoint $a \cdot e $. Hence all the bounded interval defects $\Delta_e (I,J)$ vanish. The following lemma shows that the superlevel components are bounded for almost every direction and line.

\begin{lemma}\label{almost-bounded}
    Suppose $(u,v)$ is a positive solution of \eqref{Main1}. There is a full-measure subset $ \mathcal{E}$ of $ \mathbb{S}^{N-1}$, such that for every $ e \in \mathcal{E} $ and almost every $z \in e^{\perp}, w \in e^{\perp}$, the connected components in $\mathcal{I}_{f}(e, z, t )$ and $\mathcal{I}_{g}(e, w, \tau )$ are bounded for every $t, \tau >0$.
\end{lemma}

\begin{proof}
    Define the cylinder
    \begin{align*}
        T_{e,L} = \left\{ x \in \R^N: |x-(x\cdot e)e|<  L \right\}.
    \end{align*}
   We claim that
    \begin{align}
        \frac{1}{ |\mathbb S^{N-1}| }  \int_{\mathbb S^{N-1}}\int_{T_{e,L}}f(x)\ \dx \textup{d}S(e) \leq C L^\alpha, \label{zero-defect-1}\\
        \frac{1}{| \mathbb S^{N-1}|} \int_{\mathbb S^{N-1}} \int_{T_{e,L}} g(x)\ \dx \textup{d}S(e) \leq C L^\beta. \label{zero-defect-2}
    \end{align}

    Fix $x$ with $|x|>2L$. If $x\in T_{e,L}$, then $|x|\sin\theta<L$, where $\theta$ is the angle between $x$ and $e$. Hence, every $ e \in \mathbb{S}^{N-1} $ such that $ x \in T_{e,L} $ must lie in one of two spherical caps centered at $x/|x|$ and $-x/|x|$, whose angular radii are bounded by $ \arcsin(L/|x|) \leq CL/|x|$. Using spherical coordinates, we obtain
    \begin{align}
        \frac{1}{|S^{N-1}|} \int_{S^{N-1}} \chi_{T_{e,L}}(x) \ \textup{d}S(e) & \leq \frac{2}{|S^{N-1}|} \int_{S^{N-2}} \int_0^{CL/|x|} (\sin \theta)^{N-2} \ \textup{d} \theta \textup{d}S(\omega) \label{zero-defect-0} \\
        & \leq C\min\left\{1,\left(\frac{L}{|x|}\right)^{N-1}\right\}. \nonumber
    \end{align}
    Decomposing $\R^N$ into $B_{2L}$ and dyadic annuli and using Proposition~\ref{lem-int-repre}, \eqref{zero-defect-0}, we obtain
    \begin{align*}
         \frac{1}{|\mathbb S^{N-1}|} \int_{\mathbb S^{N-1}}\int_{T_{e,L}}f(x)\ \dx \ \textup{d} S(e)  & = \int_{\R^N} f(x) \left( \frac{1}{|\mathbb S^{N-1}|} \int_{\mathbb S^{N-1}} \chi_{T_{e,L}}(x)\ \textup{d} S(e) \right) \ \dx\\
        & \leq C \int_{B_{2L}} f(x)\ \dx + C \sum_{j=1}^\infty 2^{-j(N-1)}\int_{B_{2^{j+1} L}}f(x) \ \dx \\
        &\leq C L^\alpha \left(1+\sum_{j=1}^\infty2^{-j(N-1-\alpha)}\right) \leq C L^\alpha.
    \end{align*}
    Here we used the fact that $ N-1-\alpha=\beta + 1>0 $. Hence, \eqref{zero-defect-1} holds. The estimate \eqref{zero-defect-2} follows similarly.

    From \eqref{zero-defect-1} and \eqref{zero-defect-2}, we obtain that for fixed $L>0$ and almost every direction $e$,
    \begin{align}\label{zero-defect-3}
        \int_{T_{e,L}}f(x)\ \dx + \int_{T_{e,L}} g(x)\ \dx < + \infty.
    \end{align}
    Taking the intersection of these directional sets over all positive integers $L$, we obtain a subset $\mathcal E \subset \mathbb S^{N-1}$ of full measure such that \eqref{zero-defect-3} holds for every $e\in\mathcal E$ and $L>0$. Writing $x=z+se$ and applying Fubini's theorem, we obtain
    \begin{align*}
        \int_{\R} f(z+se) \ \ds < + \infty, \ \   \int_{\R} g(z+se) \ \ds < + \infty
    \end{align*}
    for almost every $z \in e^\perp$. In conclusion, every connected component of $\{ s:f(z+se) > t \}$ is bounded whenever $t>0$. Indeed, an unbounded component $ I $ of $\{ s:f(z+se) > t \}$ would imply
    \begin{align*}
        \int_{\R}f(z+se)\ \ds \geq \int_I f(z+se)\ \ds \geq t |I|=+\infty,
    \end{align*}
    which is a contradiction. The same conclusion holds for $g$. This completes the proof.
\end{proof}

Next, for a bounded open set $\Omega\subset\R^N$, define
\begin{align}\label{D-e-Omega}
    D_e(\Omega) & = \int_{e^\perp} \int_{e^\perp} \int_0^\infty \int_0^\infty \sum_{I \in \mathcal I_f(e,z,t)} \sum_{J \in \mathcal I_g( e,w, \tau)} \chi_{\{ z+ \overline I e \subset \Omega,\; w + \overline J e \subset \Omega\}} \Delta_e (I,J) \ \text{d}\tau \dt \text{d}w \text{d}z,
\end{align}
and
\begin{align}\label{D-Omega}
    D (\Omega) & = \frac{1}{|\mathbb S^{N-1}|} \int_{\mathbb S^{N-1}}D_e(\Omega)\ \text{d} S(e).
\end{align}
Here $z+\overline I e$ denotes the segment $\{z+se:s\in\overline I\}$, and similarly for $w+\overline J e$. 

By Lemma~\ref{lem:interval-force}, we have $D_e(\Omega) \geq 0$ and $ D(\Omega) \geq 0 $. Moreover, both $D_{e}(\Omega)$ and $D(\Omega)$ vanish for two radial decreasing functions $(u,v)$ with a common center. In this section, we will also prove the converse, that is : if $D(B_R)=0$ for every $R>0$, then the solution $(u,v)$ must be radial about a common point and hence belongs to $ \mathcal{G} $, which is defined in \eqref{ground-state-orbit}.

\begin{thm}\label{thm:zero-defect}
    Suppose $(u,v)$ is a positive solution of \eqref{Main1}. If
    \begin{align*}
        D(B_R)=0,\ \ \text{for every }R>0,
    \end{align*}
    then $u$ and $v$ are radial about a common point. Consequently, $(u,v)\in\mathcal G$.
\end{thm}

\begin{proof}
    By hypothesis, $D(B_k)=0$ for every positive integer $k$. Since
    \begin{align*}
        D(B_k)=\frac1{|\mathbb S^{N-1}|} \int_{\mathbb S^{N-1}}D_e(B_k)\ \textup{d}S(e),
    \end{align*}
    the nonnegativity of $D_e(B_k)$ implies that $ D_e(B_k)=0 $ for almost every $e \in \mathbb S^{N-1}$. Taking a countable intersection and combining it with the set $\mathcal E$ in Lemma \ref{almost-bounded}, we can find a subset $\tilde{\mathcal{E}} \subset \mathbb{S}^{N-1}$ with full measure, such that for any $e \in \tilde{\mathcal{E}} $, all positive superlevel components of both $f$ and $g$ are bounded on almost every line, and $D_e(B_k)=0$ for every $k\in\mathbb N$.

    By the definition of $D_e(B_k)$, we have for $ e \in \widetilde{\mathcal{E}}$ that
    \begin{align*}
       \int_{e^\perp} \int_{e^\perp} \int_0^\infty \int_0^\infty \sum_{I\in\mathcal I_f(e,z,t)} \sum_{J \in \mathcal I_g(e,w,\tau)} \chi_{\{z+ \overline I e \subset B_k,\;w + \overline J e\subset B_k\}} \Delta_e(I,J)\ \textup{d} \tau \dt \textup{d}w \textup{d}z = 0.
    \end{align*}
    Every term in this integral is nonnegative. Moreover, every interval pair $(I, J)$ is contained in some sufficiently large ball. Thus, by monotone convergence as $k\to\infty$, $ \Delta_e(I,J)=0 $ for almost every pair $(I,J)$. Then Lemma \ref{lem:interval-force} implies $ m_I=m_J $ for almost every such pair $(I,J)$.

    We claim that these intervals have one common midpoint, depending only on $e$. By Fubini's theorem, we may choose a bounded component $I_0$ of $f$ such that $m_J=m_{I_0}$ for almost every component $J$ of $g$. Set $c_e=m_{I_0}$. Choose $J_0$ from this full-measure set and from the set of components $J$ for which $m_I=m_J$ for almost every $I$. Then $m_I=c_e$ for almost every component $I$ of $f$. Therefore, $ m_I=m_J=c_e $ for almost every interval pair $(I,J)$. In particular, for almost every $z \in e^\perp$ and almost every $ t > 0 $, every connected component of $ \{s\in\R:f(z+se)>t\} $ has midpoint $c_e$. There can be at most one such component, because two nonempty open intervals with the same midpoint must intersect. Consequently, each nonempty superlevel set is an interval of the form $ (c_e-r,c_e+r) $ for some $r>0$.

    Write $ x = z + s e$ for $ z \in e^{\perp} $, we obtain
    \begin{align}\label{zero-defect-4}
        f(z+(c_e+s)e)=f(z+(c_e-s)e)
    \end{align}
    and
    \begin{align}\label{zero-defect-5}
        f(z+(c_e+s_1)e)\geq f(z+(c_e+s_2)e), \ \  0  \leq s_1 < s_2,
    \end{align}
    for almost every $ z \in e^\perp $. The same statements hold for $g$. Since $u$ and $v$ are continuous, \eqref{zero-defect-4} and \eqref{zero-defect-5} hold for every $z\in e^\perp$. Since $f=u^q$ and $g=v^p$, with $p,q>0$, we conclude that both $u$ and $v$ are symmetric with respect to the hyperplane
    \begin{align*}
        \Pi_e=\{x\in\R^N:x\cdot e=c_e\},
    \end{align*}
    and are nonincreasing along every line perpendicular to $\Pi_e$ as the distance from $\Pi_e$ increases, for any $ e \in \tilde{\mathcal{E}}$.

    Next, we show that the monotonicity in \eqref{zero-defect-5} is strict. Fix $z\in e^\perp$. Proposition~\ref{lem-int-repre} gives
    \begin{align}\label{zero-defect-6}
        u(z+(c_e+s)e) = \int_{e^\perp} \int_{\R} K(z-w+(s-\tau)e) g(w+(c_e+\tau)e) \ \textup{d} \tau \textup{d} w.
    \end{align}
    For almost every $ w \in e^\perp $, the function $\tau \mapsto g (w+(c_e+\tau)e)$ is positive, even, and nonincreasing in $|\tau|$. Each nonempty positive superlevel set is therefore a bounded interval $(-b,b)$, where $b=b(w,t)$. For a nonempty interval and $w\ne z$, set $\rho=|z-w|$ and define
    \begin{align*}
        k_\rho(s) = c_N(\rho^2+s^2)^{-(N-2)/2}, \ \ Q_{\rho,b}(s) = \int_{-b}^b k_\rho(s-\tau)\ \text{d} \tau.
    \end{align*}
    Since $k_\rho$ is even and strictly decreasing on $(0,\infty)$, we have
    \begin{align*}
        Q_{\rho,b}'(s) = k_\rho(s+b) - k_\rho(s-b)< 0, \ \  s>0.
    \end{align*}
    Thus $Q_{\rho,b}$ is strictly decreasing for $s\geq0$ when $b>0$. Since
    \begin{align*}
        g(w + (c_e + \tau) e) = \int_0^{\infty} \chi_{\{ |\tau| < b(w, t) \}} \ \dt,
    \end{align*}
    we obtain from \eqref{zero-defect-6} that
    \begin{align}\label{zero-defect-add}
        u(z+(c_e+s)e) & = \int_{e^\perp} \int_0^\infty \int_{-b(w, t)}^{b(w, t)} k_{\rho}(s- \tau)\  \textup{d} \tau \dt \textup{d} w  = \int_{e^\perp} \int_0^\infty Q_{ \rho , b(w, t)} (s)\ \dt \textup{d} w.
    \end{align}
    Each $Q_{\rho,b}$ is nonincreasing on $[0,\infty)$, and it is strictly decreasing when $b>0$. Since $g>0$, the set of $(w,t)$ with $b(w,t)>0$ has positive measure. Hence, \eqref{zero-defect-add} implies
    \begin{align}\label{zero-defect-7}
        u(z+(c_e+s_1)e) > u(z+(c_e+s_2)e), \ \ 0 \leq s_1 < s_2,
    \end{align}
    for $e\in\tilde{\mathcal E}$. The same argument applied to $v=\mathcal I(u^q)$ gives the analogous strict inequality for $v$.

     We now choose an orthonormal basis $e_1,\ldots,e_N \in \tilde{\mathcal{E}} $. This is possible because $\tilde{\mathcal E}$ has full measure in $ \mathbb{S}^{N-1} $. Let
    \begin{align*}
        \Pi_i = \{x \in \R^N: x \cdot e_i = c_{e_i} \}, \ \ i=1,\ldots,N.
    \end{align*}
    These hyperplanes have a unique intersection point $ a=\sum_{i=1}^N c_{e_i}e_i $. We claim that $a$ is the unique global maximum point of both $u$ and $v$. To see this, suppose $ x \neq a$ and set $x^{(0)}=x$. For $i=1,\ldots,N$, define
    \begin{align*}
        x^{(i)} =x^{(i-1)} +\bigl(c_{e_i}-x^{(i-1)}\cdot e_i \bigr)e_i.
    \end{align*}
    Then $x^{(N)}=a$. At each step, the point moves along a line parallel to $e_i$ toward the hyperplane $\Pi_i$. By \eqref{zero-defect-7}, we obtain
    \begin{align*}
        u(x^{(i)})\geq u(x^{(i-1)}), \ \  v(x^{(i)})\geq v(x^{(i-1)}).
    \end{align*}
    Since $x \neq a$, at least one coordinate changes, and therefore
    \begin{align*}
        u(x)<u(a), \ \  v(x)<v(a), \ \  x\neq a.
    \end{align*}
    Thus $a$ is the unique maximum point of both $ u$ and $ v$.

    We now return to an arbitrary direction $e \in \tilde{\mathcal{E}}$. Since $u$ and $v$ are symmetric with respect to $\Pi_e$, the reflection of their common maximum point $a$ across $\Pi_e$ must also be a maximum point. Since the maximum point is unique, we deduce that $a \in \Pi_e$. Consequently, $ a\cdot e=c_e $ for every $ e \in \tilde{\mathcal{E}} $. Since $ \tilde{\mathcal{E}} $ has full measure, it is dense in $\mathbb S^{N-1}$. By continuity of $u,v$, the reflection symmetry extends to every hyperplane through $a$. Hence both components are radial about $a$.

    Finally, Lemma \ref{radial-uniqueness} implies that $(u,v)$ is a translation and critical dilation of $(U,V)$. Therefore $(u,v) \in \mathcal G $, as required.
\end{proof}

\section{Estimate of the defect}\label{sec:energy}

In this section, we apply Lemma~\ref{lem:interval-force} to estimate $D_e(\Omega)$. Since $D_e(\Omega)$ is defined using a characteristic function, we first introduce a smooth cutoff. Let $ \eta \in C_c^\infty(\R^N)$ be a radial nonincreasing function satisfying
\begin{align*}
    0 \leq \eta \leq 1, \ \  \eta=1 \ \ \text{in} \ \  B_1, \ \  \eta=0 \ \ \text{outside} \ \ B_2,
\end{align*}
and set $\eta_R(x)=\eta(x/R)$ for $R>0$. For a complete interval $I$ on the line $z+\R e$, we define
\begin{align*}
    \chi_R(I)= \inf_{s\in I} \eta_R(z+se).
\end{align*}
Then $0\leq\chi_R(I)\leq\eta_R(z+se)$ for $s\in I$. Moreover, $\chi_R(I)>0$ implies that $z+\overline I e\subset B_{2R}$.

We now define the localized defect
\begin{align}\label{D-e-R}
    D_{e,R} &= \int_{e^\perp} \int_{e^\perp} \int_0^\infty \int_0^\infty \sum_{ I \in \mathcal I_f(e,z,t)} \sum_{J \in \mathcal I_g(e,w,\tau)} \min\{\chi_R (I),\chi_R (J)\} \Delta_e(I,J) \ \text{d} \tau \dt \text{d}w \text{d}z,
\end{align}
and
\begin{align}\label{D-R}
    D_R & = \frac{1}{|\mathbb S^{N-1}|} \int_{\mathbb S^{N-1}} D_{e,R}\ \text{d} S(e).
\end{align}
A weighted defect with zero coefficient is defined to be zero, without evaluating the midpoint of an unbounded interval. The reason for taking the infimum of $\eta_R$ on each interval is that Lemma~\ref{lem:interval-force} requires constant weights on the intervals. At the same time, the Lipschitz continuity of $\eta_R$ gives
\begin{align*}
    0 \leq \eta_R(z+se)-\chi_R(I)\leq\frac{C}{R}|I|, \ \ s\in I,
\end{align*}
whenever $I$ is bounded.

We first establish the basic properties of these defects.

\begin{lemma}\label{lem:defect-properties}
    For every bounded open set $\Omega\subset\R^N$ and every $R>0$, $ D(\Omega) $ and $ D_R $ are nonnegative and finite. Moreover, the following properties hold.

    \textup{(1)}. For $a \in \R^N$ and $\lambda > 0$, define
    \begin{align*}
        \widetilde u(x)=\lambda^\alpha u(a+\lambda x), \ \  \widetilde v(x)=\lambda^\beta v(a+\lambda x),
    \end{align*}
    Then
    \begin{align*}
        D_{\widetilde u, \widetilde v}(\Omega) = D_{u,v} (a+\lambda\Omega).
    \end{align*}

    \textup{(2)}. If $\Omega_1, \ldots, \Omega_k$ are pairwise disjoint bounded open subsets of $B_R$, then
    \begin{align*}
        \sum_{i=1}^k D(\Omega_i)\leq D_R.
    \end{align*}
\end{lemma}

\begin{proof}
    The nonnegativity of $D(\Omega)$ and $D_R$ follows immediately from Lemma~\ref{lem:interval-force}. We then show that they are finite. If $\chi_R(I)$ and $\chi_R(J)$ are both positive, then $z+\overline I e$ and $w+\overline J e$ are contained in $B_{2R}$. Therefore, their midpoint difference satisfies $ |m_I-m_J| \leq 4 R $. Hence
    \begin{align*}
        \Delta_e(I,J) \leq 4R \int_{I\times J} |\partial_s k(s,s')| \ \ds \ds',
    \end{align*}
    which, together with \eqref{repre-fgK}, \eqref{D-e-R}, and $ | \partial_s k| \leq |\nabla K| $ yields
    \begin{align*}
        D_{e,R} \leq 4R \int_{B_{2R}} \int_{B_{2R}} f(x)g(y) |\nabla K(x-y)| \,dy\,dx<+\infty.
    \end{align*}
    Indeed, $f$ and $g$ are bounded on $B_{2R}$, while $|\nabla K(x)|\leq C|x|^{1-N}$ is locally integrable. This estimate is uniform in $e$, and therefore $D_R<+\infty$.

    We next show that $D(\Omega)$ is finite. Since $\eta_R=1$ in $B_R$, the weight $\min\{\chi_R(I),\chi_R(J)\}$ equals one whenever both interval closures lie in $B_R$. Besides, $ \min\{ \chi_R(I), \chi_R(J) \} > 0 $ implies that $ z + \bar{I} e, w + \bar{J}e \subset B_{2R}$. Thus
    \begin{align*}
        D(B_R) \leq D_R \leq D(B_{2R}).
    \end{align*}
    By taking $R$ sufficiently large that $ \Omega \subset B_R$, we obtain $D(\Omega) \leq D_R < +\infty$.

    For the scaling property, let $\widetilde f=\widetilde u^q$ and $\widetilde g=\widetilde v^p$. Then
    \begin{align*}
        \widetilde f(x) = \lambda^{q \alpha} f(a+\lambda x),\ \  \widetilde g(x) = \lambda^{p\beta} g(a+\lambda x).
    \end{align*}
    Now, fix $ e \in S^{N-1}, $ we write $ a = a_\perp + a_\parallel e $, where $ a_\perp \in e^\perp $ and $ a_\parallel = a \cdot e$. For $ z, w \in e^\perp $ and $ s,s' \in \mathbb R $, let
    \begin{align}\label{ZWSS'}
        Z = a_\perp + \lambda z, \ \ W=a_\perp + \lambda w, \ \  S = a_\parallel + \lambda s, \ \  S' = a_\parallel + \lambda s'.
    \end{align}
    Then
    \begin{align*}
        \widetilde f(z+se) = \lambda^{q\alpha} f(Z+Se), \ \  \widetilde g(w+s'e) = \lambda^{p\beta} g(W+S'e).
    \end{align*}
    In particular, setting $ t' = \lambda^{- q \alpha} t $ and $ \tau' = \lambda^{-p\beta} \tau $, we obtain
    \begin{align*}
        \{s: \widetilde f( z + se )>t\} & = \lambda^{-1} \bigl( \{ S: f( Z + Se ) > t'\} - a_\parallel \bigr), \\
        \{s': \widetilde g( w + s'e ) > \tau \} & = \lambda^{-1} \bigl( \{S': g(W+S'e) > \tau' \} -a_\parallel \bigr).
    \end{align*}
    Thus every interval $\widetilde I$ associated with $\widetilde f$ and $t$ corresponds to the interval $I = a_\parallel + \lambda\widetilde I$ associated with $f$ and $ \lambda^{-q \alpha} t$, and similarly $ J = a_\parallel + \lambda \widetilde J $ for $ g $ and $ \lambda^{-p \beta} \tau $. Consequently,
    \begin{align*}
        m_{\widetilde I}=\frac{m_I-a_\parallel}{\lambda}, \ \  m_{\widetilde J} = \frac{m_J-a_\parallel}{\lambda},
    \end{align*}
    and hence $ m_{\widetilde I}-m_{\widetilde J} = \lambda^{-1}(m_I-m_J) $.
    Next, for the Green function, we have
    \begin{align*}
        \widetilde k(s,s')  = K(z-w+(s-s')e) = \lambda^{N-2} K(Z-W+(S-S')e) =\lambda^{N-2}k(S,S').
    \end{align*}
    It follows that $ \partial_s\widetilde k(s,s') = \lambda^{N-1} \partial_S k(S,S') $. Thus we obtain
    \begin{align*}
        \widetilde T_e(\widetilde I,\widetilde J) & = \int_{\widetilde I\times\widetilde J} \partial_s\widetilde k(s,s')\ \ds \ds' = \lambda^{N-3} \int_{I\times J} \partial_S k(S,S')\ \text{d}S \text{d} S' = \lambda^{N-3} T_e(I,J).
    \end{align*}
    Therefore,
    \begin{align*}
        \widetilde \Delta_e( \widetilde I, \widetilde J ) = -( m_{\widetilde I} - m_{\widetilde J} ) \widetilde T_e( \widetilde I, \widetilde J) =\lambda^{N-4}\Delta_e(I,J).
    \end{align*}
    Moreover, $z+\overline{\widetilde I}e\subset\Omega$ if and only if $Z+\overline I e\subset a+\lambda\Omega$, and similarly for $J$. Then, from \eqref{ZWSS'} and $ t' = \lambda^{-q \alpha} t$, $ \tau' = \lambda^{-p \beta} \tau $, we have
    \begin{align*}
        D_e^{\widetilde u,\widetilde v}(\Omega) = \lambda^{N-4 + q\alpha + p \beta- 2(N-1)} D_e^{u,v}(a + \lambda \Omega )  = \lambda^{q \alpha + p \beta- N- 2} D_e^{u,v}(a + \lambda \Omega) =  D_e^{u,v}(a + \lambda \Omega).
    \end{align*}

    Finally, suppose that $\Omega_1,\ldots,\Omega_k\subset B_R$ are pairwise disjoint. Then
    \begin{align*}
        \sum_{i=1}^k \chi_{\{z+ \overline I e\subset \Omega_i, \;w+\overline J e \subset\Omega_i\}} \leq 1.
    \end{align*}
    Moreover, for every such interval $ I, J$, we have $ z + \bar{I} e, w + \bar{J} e \subset \Omega_i\subset B_R$ for some $i$, and hence $ \min\{ \chi_R (I), \chi_R(J) \} =1 $, which implies
    \begin{align*}
        \sum_{i=1}^k D(\Omega_i)\leq D_R.
    \end{align*}
    This completes the proof.
\end{proof}

To estimate $D(B_R)$, we estimate $D_R$.

\begin{prop}\label{prop:defect-energy}
    Suppose $(u,v)$ is a positive solution of \eqref{Main1}. For every fixed
    \begin{align*}
        0<\sigma<\min\{1,\alpha,\beta\},
    \end{align*}
    there exists a constant $C>0$, depending only on the equation parameters and $\sigma$, such that
    \begin{align*}
        D_R\leq C (E(2R)^{1-\sigma/N} + E(2R)^{1-1/N} ), \ \  R>0,
    \end{align*}
    where
    \begin{align*}
        E(R)=\int_{B_R}\left(u^{q+1}+v^{p+1}\right)\ \dx.
    \end{align*}
\end{prop}

We first prove a lemma used in Proposition~\ref{prop:defect-energy}.

\begin{lemma}\label{E-R-bound}
    Suppose $(u,v)$ is a positive solution of \eqref{Main1}, and define
    \begin{align*}
        C_R & = \int_{\R^N} \int_{\R^N} |\eta_R(x) - \eta_R(y)| f(x) g(y) K(x-y) \ \dx \dy, \\
        L_{R} & =  \frac{1}{|\mathbb{S}^{N-1}|} \int_{\mathbb{S}^{N-1}} ( L_{f,R}(e) + L_{g,R}(e) ) \ \textup{d} S(e),
    \end{align*}
    where $ f = u^q$, $ g = v^p $ and
    \begin{align*}
        L_{f,R}(e) & = \int_{\R^N} u(x) \int_0^{f(x)} \bigl( \eta_R(x) - \chi_R (I_{e,x,t}) \bigr)\ \dt \dx, \\
        L_{g,R}(e) & = \int_{\R^N} v(y) \int_0^{g(y)} \bigl( \eta_R (y) - \chi_R (J_{e, y, \tau} ) \bigr) \ \textup{d}\tau \dy,
    \end{align*}
    Here $x=z+se$, $y=w+s'e$, $I_{e,x,t}$ is the component of $\mathcal I_f(e,z,t)$ containing $s$, and $J_{e,y,\tau}$ is the component of $\mathcal I_g(e,w,\tau)$ containing $s'$. Then
    \begin{align*}
        D_R \leq \left( 1 + \frac{1}{2} \left| \frac{1}{q+1} - \frac{1}{p+1} \right| \right) C_R + 2 L_R.
    \end{align*}
\end{lemma}

\begin{proof}
    Fix $R>0$ and fix $e \in \mathcal{E} $, which is a full-measure subset of $ \mathbb{S}^{N-1} $ defined in Lemma~\ref{almost-bounded}. Set
    \begin{align*}
        E_{f,R}=\int_{\R^N}\eta_Ru^{q+1}\ \dx,\ \
        E_{g,R}=\int_{\R^N}\eta_Rv^{p+1}\ \dx.
    \end{align*}
    Then $ E_{f,R} $ and $E_{g,R}$ are finite. From Proposition \ref{lem-int-repre} and $ 0 \leq \chi_R(I) \leq \eta_R(x) $ for $ x \in z+Ie$, we get
    \begin{align*}
        C_R\leq E_{f,R}+E_{g,R},\ \
        0\leq L_{f,R}(e)\leq E_{f,R},\ \
        0\leq L_{g,R}(e)\leq E_{g,R}.
    \end{align*}
   
    For bounded components $I\in\mathcal I_f(e,z,t)$ and $J\in\mathcal I_g(e,w,\tau)$, apply Lemma \ref{lem:interval-force} with $\gamma_I=\chi_R(I)$ and $\gamma_J = \chi_R(J) $ to obtain
    \begin{align}\label{energy-lem-1}
        & \gamma_I \int_{I\times J} (k-k_I)\ \ds \ds' + \gamma_J \int_{I \times J} (k-k_J)\ \ds \ds' + \min\{ \gamma_I, \gamma_J \} \Delta_e( I, J) \nonumber \\
        & \quad \leq d \min\{ \gamma_I, \gamma_J \} \int_{I \times J} \Theta_e(x,y) k(s,s') \ \ds \ds' + |\gamma_I-\gamma_J| \int_{I \times J} k(s,s')\ \ds \ds',
    \end{align}
    where $x=z+se$, $y=w+s'e$ and
    \begin{align*}
        \Theta_e(x,y)=\frac{((x-y)\cdot e)^2}{|x-y|^2}.
    \end{align*}

    We first consider the term involving $\gamma_I(k-k_I)$. Whenever $\chi_R(I)>0$, the interval $I=(s_-,s_+)$ is bounded. Since $I$ is a component of $ \{s:f(z+se)>t\} $ and $f=u^q$, its endpoints satisfy
    \begin{align*}
        u(z+s_-e)=u(z+s_+e)=t^{1/q}.
    \end{align*}
    
    By the representation of $g$ and $u=\mathcal I g$, we have
    \begin{align*}
        &\quad \int_{e^\perp}\int_0^\infty \sum_{J\in\mathcal{I}_g(e,w,\tau)} \int_J(k(s,s')-k_I(s'))\ \ds' \textup{d} \tau \textup{d} w \\
         & = \int_{e^{\perp}} \int_{\R} (k(s , s') - k_{I}(s')) \int_0^{\infty} \sum_{J } \chi_{J}(s') \ \textup{d}\tau \ds' \textup{d}w \\
        & = \int_{e^{\perp}} \int_{\R} (k(s , s') - k_{I}(s')) g(w + s'e) \ \ds' \textup{d}w  \\
        & = u(z+se)-\frac{u(z+s_-e)+u(z+s_+e)}{2}\\
        &= u(z+se)-t^{1/q}.
    \end{align*}
    Consequently, for the term involving $\gamma_I (k-k_I)$, we obtain
    \begin{align}\label{energy-lem-2}
        & \ \  \int_{e^{\perp}}\int_{e^{\perp}} \int_{0}^{\infty}\int_{0}^{\infty} \sum_{I \in \mathcal{I}_f (e, z, t)} \sum_{J \in \mathcal{I}_g (e, w, \tau)} \int_{I \times J} \chi_R(I) (k(s,s') - k_I (s')) \ \ds' \ds \textup{d}\tau \dt \textup{d}w \textup{d}z \nonumber \\
        & = \int_{e^{\perp}} \int_{0}^{\infty} \sum_{I \in \mathcal{I}_f (e, z, t)} \int_I \chi_R(I) (u(z + se) - t^{1/q}) \ \ds \dt \textup{d}z \nonumber \\
        & = \int_{e^{\perp}} \int_{\R} \int_0^{\infty} \sum_{I \in \mathcal{I}_f(e,z,t) } \chi_R(I) \chi_I (s) (u(z + se) - t^{1/q}) \ \dt \ds \textup{d}z \nonumber \\
        &  = \int_{\mathbb R^N}\int_0^{f(x)} \chi_R (I_{e,x,t}) (u(x)-t^{1/q})\ \dt \dx \\
        & = \int_{\R^N} \eta_R(x) \int_{0}^{f(x)} (u(x) - t^{1/q}) \ \dt \dx - \int_{\R^N} \int_{0}^{f(x)} (\eta_R(x) - \chi_R(I_{e,x,t}) ) (u(x) - t^{1/q}) \ \dt \dx \nonumber \\
        & = \frac{1}{q+1} \int_{\R^N} \eta_R (x) u^{q+1}(x) \ \dx - L_{f,R}(e) + \int_{\R^N} \int_{0}^{f(x)} (\eta_R(x) - \chi_R(I_{e,x,t}) ) t^{1/q} \ \dt \dx \nonumber \\
        & \geq \frac{1}{q+1} E_{f,R} - L_{f,R}(e), \nonumber
    \end{align}
    The last inequality uses $\eta_R(x)-\chi_R(I_{e,x,t})\geq0$. Similarly, we obtain
    \begin{align}\label{energy-lem-3}
         \int_{e^{\perp}}\int_{e^{\perp}} \int_{0}^{\infty}\int_{0}^{\infty} \sum_{I \in \mathcal{I}_f (e, z, t)} & \sum_{J \in \mathcal{I}_g (e, w, \tau)} \int_{I \times J} \chi_R(J) (k(s,s') - k_J (s))  \ \ds' \ds \textup{d}\tau \dt \textup{d}w \textup{d}z \nonumber \\
         & \geq \frac{1}{p+1}E_{g,R}-L_{g,R}(e).
    \end{align}
    Summing the third term on the left-hand side of \eqref{energy-lem-1} over $I,J$ and integrating in $z,w,t,\tau$ gives $D_{e,R}$.

    We next consider the terms on the right-hand side of \eqref{energy-lem-1}. Since
    \begin{align*}
        \min\{\gamma_I,\gamma_J\} \leq\frac{ \eta_R(x) + \eta_R(y) }{2},
    \end{align*}
    we have
    \begin{align}\label{energy-lem-4}
         & \quad d \int_{e^{\perp}} \int_{e^{\perp}} \int_{0}^{\infty}\int_{0}^{\infty} \sum_{I \in \mathcal{I}_f (e, z, t)} \sum_{J \in \mathcal{I}_g (e, w, \tau)} \int_{I \times J} \min\{ \gamma_I, \gamma_J \} \Theta_e (x,y) k(s,s')  \ \ds' \ds \textup{d}\tau \dt \textup{d}w \textup{d}z \nonumber \\
         & \leq  \frac{d}{2}  \int_{e^{\perp}} \int_{e^{\perp}} \int_{0}^{\infty} \int_{0}^{\infty} \sum_{I} \sum_{J} \int_{I \times J} (\eta_R(x) + \eta_R(y) ) \Theta_e (x,y) k(s,s')  \ \ds' \ds \textup{d}\tau \dt \textup{d} w \textup{d}z  \\
         & = \frac{d}{2}  \int_{e^{\perp}} \int_{e^{\perp}} \int_{\R} \int_{\R}  \int_{0}^{\infty} \int_0^{\infty} \sum_{I} \sum_{J } \chi_I(s) \chi_J(s') (\eta_R(x) + \eta_R(y) ) \Theta_e (x,y) k(s,s')  \ \textup{d}\tau \dt \ds' \ds  \textup{d} w \textup{d}z \nonumber  \\
        & =   \frac{d}{2}\int_{\mathbb R^N}\int_{\mathbb R^N} (\eta_R(x)+\eta_R(y))\Theta_e(x,y) f(x)g(y) K(x-y) \ \dx \dy. \nonumber
    \end{align}

    For the last term on the right-hand side of \eqref{energy-lem-1}, we use
    \begin{align*}
        |\gamma_I-\gamma_J| & \leq |\eta_R(x)-\eta_R(y)| +(\eta_R(x) - \chi_R(I)) + (\eta_R(y)-\chi_R(J)).
    \end{align*}
    As in \eqref{energy-lem-2} and \eqref{energy-lem-4}, we obtain
    \begin{align}\label{energy-lem-5}
        & \quad \int_{e^{\perp}} \int_{e^{\perp}} \int_{0}^{\infty}\int_{0}^{\infty} \sum_{I \in \mathcal{I}_f (e, z, t)} \sum_{J \in \mathcal{I}_g (e, w, \tau)} \int_{I \times J} |\gamma_I - \gamma_J| k(s,s')  \ \ds' \ds \textup{d}\tau \dt \textup{d}w \textup{d}z \nonumber \\
        & \leq \int_{\R^N} \int_{\R^N} |\eta_R (x) - \eta_R(y) | f(x) g(y) K(x-y) \ \dx \dy + L_{f,R}(e) + L_{g,R}(e) \\
        & = C_R + L_{f,R}(e) + L_{g,R}(e). \nonumber
    \end{align}
    From \eqref{energy-lem-1}-\eqref{energy-lem-5}, we obtain
    \begin{align}\label{energy-lem-6}
        \frac{1}{q+1}E_{f,R} + \frac{1}{p+1} E_{g,R} + D_{e,R} & \leq \frac{d}{2} \int_{\mathbb R^N} \int_{\mathbb R^N} (\eta_R(x) + \eta_R(y) ) \Theta_e(x,y) f(x) g(y) K(x-y) \ \dx \dy \nonumber \\
        & \quad + C_R + 2 L_{f,R}(e) + 2 L_{g,R}(e).
    \end{align}
    Now, we take average over $ e \in \mathbb S^{N-1}$ on both sides of \eqref{energy-lem-6}. For $ x \ne y$, direct computation gives
    \begin{align*}
        \frac{1}{ |\mathbb{S}^{N-1}| }\int_{\mathbb{S}^{N-1}}\Theta_e(x,y)\ \text{d} S(e)=\frac{1}{N}.
    \end{align*}
    From Proposition~\ref{lem-int-repre}, we get
    \begin{align*}
        \int_{\mathbb R^N} \int_{\mathbb R^N} \eta_R(x) f(x) g(y) K(x-y) \ \dx \dy = \int_{\R^N} \eta_R(x) f(x) u(x) \ \dx  =E_{f,R}, \\
        \int_{\mathbb R^N} \int_{\mathbb R^N} \eta_R(y) f(x) g(y) K(x-y) \ \dx \dy =  \int_{\R^N} \eta_R(y) g(y) v(y) \ \dy  = E_{g,R}.
    \end{align*}
    Thus we obtain from \eqref{energy-lem-6} that
    \begin{align*}
        \frac{1}{q+1} E_{f,R} + \frac{1}{p+1} E_{g,R} + D_R \leq \frac{N-2}{2N}(E_{f,R}+E_{g,R}) + C_R + 2 L_R.
    \end{align*}
    From \eqref{critical-hyperbola}, we deduce that
    \begin{align*}
        D_R &\leq \frac{1}{2} \left(\frac{1}{p+1}-\frac{1}{q+1}\right) (E_{f,R}-E_{g,R})+C_R+2L_R.
    \end{align*}
    On the other hand,
    \begin{align*}
        E_{f,R}-E_{g,R} = \int_{\mathbb R^N} \int_{\mathbb R^N} (\eta_R(x)-\eta_R(y)) f(x) g(y) K(x-y) \ \dx \dy ,
    \end{align*}
    and therefore $ |E_{f,R}-E_{g,R}| \leq C_R $. Consequently, the desired conclusion holds.
\end{proof}

Next, we use Lemma \ref{E-R-bound} to prove Proposition \ref{prop:defect-energy}.

\begin{proof}[Proof of Proposition \ref{prop:defect-energy}]
    We first estimate $L_R$. Recall that $\chi_R(I)$ is the infimum of $\eta_R$ in $I$. Since $\eta_R$ is Lipschitz with constant $C/R$, we have
    \begin{align}\label{energy-prop-0}
        0 \leq \eta_R(x) - \chi_R(I_{e,x,t}) \leq \frac C R|I_{e,x,t}|
    \end{align}
    whenever $I_{e,x,t}$ is bounded. The inequality also holds when the interval is unbounded. Therefore, to control $L_R$, we first estimate the average length of the intervals in $\mathcal{I}_f(e,z,t) $.

    Fix $x\in\R^N$, $0<t<f(x)$, and $\ell>0$. If $|I_{e,x,t}|>\ell$, then at least one of the two intervals of length $\ell/2$ starting from $x$ in the directions $e$ and $-e$ is contained in $\{ f > t \}$. Let $F\subset\mathbb S^{N-1}$ be the set of directions $e$ such that $x+re\in\{f>t\}$ for every $0<r<\ell/2$. Then $ |F| \leq |\mathbb{S}^{N-1}|$ and
    \begin{align}\label{energy-prop-1}
        \int_{B_{\ell/2}(x)}f(y)\,dy & \geq \int_F \int_0^{\ell/2} f(x+re)r^{N-1}\ \textup{d}r \textup{d} S(e) \geq t \int_F\int_0^{\ell/2}r^{N-1}\ \textup{d} r \textup{d} S(e) \nonumber \\
        &=\frac{t}{N} \left(\frac{\ell}{2}\right)^N |F|.
    \end{align}
    On the other hand, from Proposition~\ref{lem-int-repre}, we get
    \begin{align*}
        \int_{B_{\ell/2}(x)}f(y)\,dy\leq C\ell^\alpha,
    \end{align*}
    which combined with \eqref{energy-prop-1} gives
    $$
        |F| \leq \frac{C}{t\ell^{N-\alpha}}.
    $$
    Since the set of directions along which $|I_{e,x,t}|>\ell$ holds is contained in $F \cup (-F)$, we obtain
    \begin{align}\label{energy-prop-add}
        |\{e\in\mathbb S^{N-1}:|I_{e,x,t}|>\ell\}|\leq 2|F|  \leq \min\left\{ |\mathbb{S}^{N-1}|,\frac{C}{t\ell^{N-\alpha}}\right\}.
    \end{align}
    We now integrate \eqref{energy-prop-add} with respect to $\ell$. Since $N-\alpha=\beta+2 > 1$, we obtain
    \begin{align}\label{energy-prop-2}
        \frac{1}{|\mathbb S^{N-1}|} \int_{\mathbb S^{N-1}} |I_{e,x,t}| \ \text{d}S(e) & = \frac{1}{ |\mathbb S^{N-1}| } \int_0^\infty | \{ e: |I_{e,x,t}| > \ell\} | \ \text{d}\ell  \\
        & \leq \frac{1}{{ |\mathbb S^{N-1}| }} \int_0^\infty\min\left\{|\mathbb{S}^{N-1}| ,\frac{C}{t\ell^{N-\alpha}}\right\}\ \text{d} \ell \leq Ct^{-1/(N-\alpha)}. \nonumber
    \end{align}
    In particular, $I_{e,x,t}$ is bounded for almost every direction $e$. For $g$, we similarly obtain
    \begin{align}\label{energy-prop-3}
        \frac{1}{ | \mathbb S^{N-1}| } \int_{ \mathbb S^{N-1}} |J_{e,y,\tau}| \ \textup{d} S(e) \leq C \tau^{-1/(N-\beta)}.
    \end{align}
    Hence, using \eqref{energy-prop-0}, \eqref{energy-prop-2} and \eqref{energy-prop-3}, we obtain
    \begin{align}\label{energy-prop-4}
        L_R & \leq \frac C R \int_{B_{2R}} u(x) \int_0^{f(x)} \left( \frac{1}{|\mathbb S^{N-1}|} \int_{\mathbb S^{N-1}} |I_{e,x,t}| \ \text{d}S(e) \right) \dt \dx \nonumber \\
        & \quad + \frac{C}{R} \int_{B_{2R}} v(y) \int_0^{g(y)} \left( \frac{1}{|\mathbb S^{N-1}|} \int_{ \mathbb S^{N-1}} |J_{e, y, \tau} | \ \textup{d} S(e) \right) \textup{d}\tau \dy \nonumber \\
        & \leq \frac{C}{R} \int_{B_{2R}} \left( u(x) \int_0^{u(x)^q} t^{-1/(N-\alpha)} \ \dt + v(x) \int_0^{v(x)^p} \tau^{-1/(N-\beta)} \ \textup{d} \tau \right) \dx \nonumber \\
        & = \frac{C}{R} \int_{B_{2R}} \left(u^{(N-1)/\alpha} + v^{(N-1)/\beta}\right) \ \dx  \\
        & \leq  C \frac{|B_{2R}|^{1/N}}{R} \left( \int_{B_{2R}} u^{q+1} \right)^{1-1/N} +  C \frac{|B_{2R}|^{1/N}}{R} \left( \int_{B_{2R}} v^{p+1} \right)^{1-1/N} \nonumber \\
        & \leq CE(2R)^{1-1/N}, \nonumber
    \end{align}
    where we have used H\"older's inequality and $ \alpha(q+1) = N$, $ \beta(p+1) = N$ in the last two lines.

    Next, we estimate $C_R$. Recall that
    \begin{align*}
        C_R = \int_{\R^N}\int_{\R^N} |\eta_R(x)-\eta_R(y)|f(x)g(y)K(x-y)\ \dx \dy.
    \end{align*}
    Fix $0<\sigma<\min\{1,\alpha,\beta\}$. We first show that
    \begin{align}\label{energy-prop-5}
        \int_{\R^N}|x-y|^\sigma K(x-y)g(y)\ \dy \leq C u(x)^{1-\sigma/\alpha}.
    \end{align}
    For any $\rho>0$, we split the integral into the regions $|x-y|<\rho$ and $|x-y|\geq\rho$. From $u= \mathcal{I}g = \mathcal{I} v^p$ and Proposition \ref{lem-int-repre}, we get
    \begin{align}\label{energy-prop-add2}
        \int_{\R^N} |x-y|^\sigma K(x-y) g(y) \ \dy & =
        \int_{|x-y|<\rho} |x-y|^\sigma K(x-y)g(y) \ \dy \nonumber \\
        &\quad + \int_{|x-y|\geq\rho} |x-y|^\sigma K(x-y) g(y)\ \dy  \nonumber \\
        & \leq \rho^\sigma u(x) + C\sum_{j=0}^\infty (2^j\rho)^{\sigma-(N-2)} \int_{B_{2^{j+1}\rho}(x)}g(y)\ \dy \\
        & \leq \rho^{\sigma} u(x) + C \sum_{j=0}^\infty(2^j\rho)^{\sigma-(N-2)+\beta} \nonumber  \\
        & \leq  \rho^{\sigma} u(x) + C \rho^{ \sigma - \alpha}. \nonumber
    \end{align}
    The last series converges because $ \sigma < \alpha = N - 2 - \beta $. Choosing $\rho=u(x)^{-1/\alpha}$ in \eqref{energy-prop-add2}, we conclude that \eqref{energy-prop-5} holds. Similarly, using $v= \mathcal{I} f$ and Proposition \ref{lem-int-repre}, we obtain
    \begin{align}\label{energy-prop-6}
        \int_{\R^N}|x-y|^\sigma K(x-y)f(x)\,dx \leq Cv(y)^{1-\sigma/\beta}.
    \end{align}
    We now return to $C_R$. Since $\eta_R$ is smooth and $0<\sigma<1$, we have
    \begin{align*}
        |\eta_R(x)-\eta_R(y)| \leq CR^{-\sigma}|x-y|^\sigma.
    \end{align*}
    Moreover, this difference vanishes whenever both $x$ and $y$ lie outside $B_{2R}$. By \eqref{energy-prop-5} and \eqref{energy-prop-6}, we obtain
    \begin{align}\label{energy-prop-7}
        C_R & \leq C R^{-\sigma}\int_{B_{2R}} f(x) \left(\int_{\R^N} |x-y|^\sigma K(x-y)g(y)\ \dy \right) \dx \nonumber \\
        & \quad + C R^{-\sigma} \int_{B_{2R}} g(y) \left( \int_{\R^N} |x-y|^\sigma K(x-y)f(x)\ \dx \right) \dy  \nonumber \\
        &\leq CR^{-\sigma}\int_{B_{2R}} \left( u^{q+1-\sigma/\alpha} + v^{p+1-\sigma/\beta} \right)dx \\
        & \leq C R^{-\sigma} |B_{2R}|^{\sigma/N} \left( \int_{B_{2R}} u^{N/\alpha} \right)^{1-\sigma/N} + C R^{-\sigma} |B_{2R}|^{\sigma/N} \left( \int_{B_{2R}} v^{N/\beta} \right)^{1-\sigma/N} \nonumber \\
        & \leq C  E(2R)^{1-\sigma/N}. \nonumber
    \end{align}
    Hence, from Lemma \ref{E-R-bound}, \eqref{energy-prop-4}, \eqref{energy-prop-7}, we have
    \begin{align*}
        D_R \leq CC_R+CL_R \leq CE(2R)^{1-\sigma/N}+CE(2R)^{1-1/N},
    \end{align*}
    as required.
\end{proof}

The preceding proposition provides an estimate for $D_R$ in terms of $E(2R)$. However, this estimate alone does not imply the vanishing of the defect $D(B_R)$. The following theorem shows that, for a bounded positive solution of \eqref{Main1}, $E(R)=o(R)$ implies $D(B_R)=0$; hence the solution is a translation and dilation of the ground state.

\begin{thm}\label{thm:sublinear-energy}
    Let $(u,v)$ be a bounded positive entire solution of \eqref{Main1}. If
    \begin{align*}
        E(R) = o (R), \ \  \text{as } \ \ R \to \infty,
    \end{align*}
    then $u$ and $v$ are radial about a common point. Consequently, $(u,v) \in \mathcal G$.
\end{thm}

\begin{proof}
    From Theorem \ref{thm:zero-defect}, we only need to prove that
    \begin{align}\label{E-oR}
        D(B_L) = 0, \ \ \text{for all} \ \ L > 0.
    \end{align}
    In order to prove \eqref{E-oR}, we show that $\Delta_e(I,J) = -(m_I - m_J) T_e(I,J) = 0 $ for almost every interval pair $(I,J)$ and $e \in \mathbb{S}^{N-1}$.

    By the proof of Lemma~\ref{almost-bounded}, there is a full-measure set $\mathcal E\subset\mathbb S^{N-1}$ such that, for every $e\in\mathcal E$,
    \begin{align}\label{E-oR-0}
        \int_{T_{e,L}}(f+g)\ \dx< \infty, \ \ f = u^q, \ \ g = v^p, \ \  T_{e,L} = \{ x \in \mathbb R^N: |x-(x\cdot e)e|< L\},
    \end{align}
    for all $L>0$. Fix $ e \in \mathcal E $ and a constant $ M > 0 $. Choose a smooth nonincreasing function $ \zeta : [0,\infty) \to [0,1]$ satisfying $ \zeta = 1 $ on $[0,1]$ and $\zeta = 0$ on $[2,\infty)$, and set $ \zeta_R (z) = \zeta( |z|/R ) $ for $ z \in e^\perp $. Define
    \begin{align*}
        A_{e,R} & = -\frac{1}{2} \int_{e^\perp} \int_{e^\perp} \int_0^\infty \int_0^\infty \sum_{I } \sum_{J } \bigl( \zeta_R(z) + \zeta_R(w) \bigr) \bigl( \vartheta_M (m_I) - \vartheta_M(m_J) \bigr) T_e(I,J) \ \textup{d} \tau \dt \textup{d} w \textup{d} z,\\
        B_{e,R} & = \frac{1}{2} \int_{e^\perp} \int_{e^\perp} \int_0^\infty \int_0^\infty \sum_{I } \sum_{J } \bigl( \zeta_R(z) - \zeta_R(w) \bigr) \bigl( \vartheta_M(m_I) + \vartheta_M(m_J) \bigr) T_e(I,J) \ \textup{d} \tau \dt \textup{d} w \textup{d}z,
    \end{align*}
    where the sums are taken over $I\in\mathcal I_f(e,z,t)$ and $J\in\mathcal I_g(e,w,\tau)$, and
    \begin{align*}
        \vartheta_M (s) = \begin{cases}
            -M, \ \ & s < -M, \\
            s, \ \ & -M \leq s \leq M, \\
            M, \ \ & s > M.
        \end{cases}
    \end{align*}
    Next, we show that the integrals defining $A_{e,R}$ and $B_{e,R}$ converge absolutely. Since $\zeta_R$ has compact support and $\vartheta_M$ is bounded, it suffices to integrate $|T_e(I,J)|$ with either $z$ or $w$ restricted to a bounded set.

    Fix an interval $ I=(s_-,s_+) \in \mathcal I_f(e,z,t) $. Write $ x_\pm = z + s_\pm e $. Then $ u(x_-) = u(x_+) = t^{1/q} $. For $ J \in \mathcal I_g(e,w,\tau)$, integration with respect to $s$ yields
    \begin{align*}
        T_e(I,J) = \int_{I \times J} \partial_s k(s, s') \ \ds \textup{d}s' = \int_J \bigl( K(x_+-w-s'e) - K(x_- -w-s'e) \bigr)\ \ds'.
    \end{align*}
    Thus we obtain
    \begin{align}\label{E-oR-1}
        & \ \  \int_{e^\perp}\int_0^\infty \sum_{J\in\mathcal I_g(e,w,\tau)} T_e(I,J)\ \textup{d} \tau \textup{d} w \nonumber\\
        & = \int_{e^{\perp}} \int_{0}^{\infty} \sum_{J\in\mathcal I_g(e,w,\tau)} \int_{\R} (K(x_+ - w - s' e) - K(x_- - w - s'e)) \chi_{J}(s') \ \ds' \textup{d} \tau \textup{d} w \nonumber \\
        & = \int_{e^{\perp}} \int_{\R}  (K(x_+ - w - s' e) - K(x_- - w - s'e))  \int_{0}^{\infty} \sum_{J\in\mathcal I_g(e,w,\tau)} \chi_{J}(s') \ \textup{d} \tau \ds'  \textup{d} w  \nonumber \\
        & = \int_{\mathbb R^N}\bigl(K(x_+-y)-K(x_--y)\bigr) g(y)\ \dy \\
        & =u(x_+)-u(x_-)=0. \nonumber
    \end{align}
   Similarly, we have
    \begin{align}\label{E-oR-2}
        \int_{e^\perp} \int_0^\infty \sum_{ J \in \mathcal I_g( e,w,\tau)} |T_e(I,J)|\ \textup{d} \tau \textup{d} w \leq \int_{\mathbb R^N} \bigl( K(x_+-y) + K(x_- - y) \bigr) g(y)\ \dy = 2 t^{1/q}.
    \end{align}
    For almost every $z$, all positive superlevel components are bounded and $s\mapsto f(z+se)$ is $C^1$ on each compact interval. The set of its critical values has measure zero. At every positive regular level, each level point is an endpoint of exactly one superlevel component, and every component has two endpoints. Summing \eqref{E-oR-2} over $I$ and integrating in $t$, we may therefore apply the one-dimensional area formula to obtain
    \begin{align}
        & \ \  \int_0^{\infty} \int_{e^\perp} \int_0^\infty \sum_{ I \in \mathcal I_f(e, z, t)} \sum_{ J \in \mathcal I_g(e,w,\tau)} |T_e(I,J)|\ \textup{d} \tau \textup{d} w  \dt  \nonumber \\
        & \leq 2\int_0^{\infty} \sum_{ I \in \mathcal I_f(e, z, t)} t^{1/q} \ \dt = 2 \int_0^{\infty} t^{1/q} \# \mathcal{I}_f(e,z,t) \ \dt \label{E-oR-add} \\
        & = \int_0^{\infty} t^{1/q} \# \{s: f(z + se ) =t \} \ \dt  = \int_{\R} f(z+se)^{1/q} |\partial_e f(z + se)| \ \ds \nonumber \\
        & = q\int_{\R} f(z+se) |\partial_e u(z + se)| \ \ds, \nonumber
    \end{align}
    The last equality uses $ \partial_e f=q u^{q-1}\partial_e u $. Since $(u,v)$ is bounded, Proposition \ref{lem-int-repre} gives $ \|\nabla u\|_{L^\infty}+\|\nabla v\|_{L^\infty}<\infty$. Thus, by \eqref{E-oR-0} and \eqref{E-oR-add}, we obtain
    \begin{align}\label{E-oR-3}
        &\quad \int_{\substack{z\in e^\perp \\ |z|<L} } \int_{e^\perp} \int_0^\infty \int_0^\infty \sum_I \sum_J |T_e(I,J)| \ \textup{d} \tau \dt \textup{d} w \textup{d} z \\
        &  \leq q\int_{T_{e,L}} f(x)|\partial_e u(x)| \dx \\
        & \leq q\|\nabla u\|_{L^\infty} \int_{T_{e,L}} f(x)\ \dx < \infty. \nonumber
    \end{align}
    The same argument applies to $g=v^p$, which is $C^1$ on compact subsets even when $p<1$, since $v>0$. Using $\partial_s k=-\partial_{s'}k$, for a fixed bounded component $J=(s'_-,s'_+)\in\mathcal I_g(e,w,\tau)$ we obtain
    \begin{align}
        \int_{e^\perp}\int_0^\infty \sum_{I \in \mathcal I_f(e,z,t)} T_e(I,J) \ \dt \textup{d} z & = v(w+s'_-e) - v(w+s'_+e) = 0, \label{E-oR-4}\\
        \int_{\substack{w\in e^\perp \\ |w|<L}} \int_{e^\perp} \int_0^\infty \int_0^\infty \sum_I \sum_J |T_e(I,J)| \ \dt \textup{d}\tau \textup{d} z \textup{d} w & \leq p\|\nabla v\|_{L^\infty} \int_{T_{e,L}} g(x) \ \dx < \infty. \label{E-oR-5}
    \end{align}
    Hence, $A_{e, R}$ and $B_{e,R}$ are well-defined for every $e \in \mathcal{E}$ and every $R>0$. Let
    \begin{align*}
        \Delta_e^M(I,J)=-(\vartheta_M(m_I)-\vartheta_M(m_J))T_e(I,J).
    \end{align*}
    Lemma \ref{lem:interval-force} implies that $\Delta_e^M(I,J)  \geq 0 $. By the definition of $\vartheta_M$, $\Delta_e^M (I,J) \to \Delta_e (I,J)$ as $M \to +\infty$.

    Multiplying \eqref{E-oR-1} by $ \zeta_R (z) \vartheta_M ( m_I ) $ and integrating over $(z, t, I)$, and multiplying \eqref{E-oR-4} by $ \zeta_R(w) \vartheta_M (m_J) $ and integrating over $(w, \tau, J)$, we obtain that
    \begin{align}\label{E-oR-add2}
        0 = \int_{e^\perp} \int_{e^\perp} \int_0^\infty \int_0^\infty \sum_I \sum_J \bigl( \zeta_R(z) \vartheta_M (m_I) - \zeta_R(w) \vartheta_M (m_J) \bigr) T_e(I,J) \ \textup{d} \tau \dt \textup{d} w \textup{d} z.
    \end{align}
    Using the identity
    \begin{align*}
        &\quad \zeta_R(z) \vartheta_M (m_I) -\zeta_R(w) \vartheta_M(m_J)\\
          & = \frac{1}{2} \bigl( \zeta_R(z) + \zeta_R(w) \bigr) \bigl( \vartheta_M (m_I) - \vartheta_M (m_J) \bigr) \\
        &\quad + \frac{1}{2} \bigl( \zeta_R(z) -\zeta_R(w) \bigr) \bigl( \vartheta_M (m_I) + \vartheta_M (m_J) \bigr),
    \end{align*}
    we conclude from \eqref{E-oR-add2} and the definition of $A_{e,R}$ and $B_{e,R}$ that $ 0 \leq A_{e,R}=B_{e,R} $. To estimate $B_{e,R}$, note that
    \begin{align*}
        |\zeta_R(z) - \zeta_R(w)| \leq\frac{C}{R}|z-w| \bigl( \chi_{\{|z| < 2R\}} + \chi_{\{|w|<2R\}} \bigr),
    \end{align*}
   and for $x=z+se$ and $y=w+s'e$, it holds that
    \begin{align*}
        |z-w|\,|\partial_s k(s,s')| = (N-2)\frac{|z-w|\,|s-s'|}{|z-w|^2+(s-s')^2} K(x-y) \leq \frac{N-2}{2}K(x-y).
    \end{align*}
    Consequently,
    \begin{align*}
        |z-w|\, |T_e(I,J)| \leq \frac{N-2}{2} \int_{I \times J} k(s,s')\ \ds \ds'.
    \end{align*}
    Hence, for $e \in \mathcal{E}$ and any $ R > 0$, we have
    \begin{align}\label{E-oR-6}
        & 0 \leq A_{e,R}=B_{e,R} \nonumber\\
        & \quad \leq \frac{CM}{R} \int_{e^\perp} \int_{e^\perp} \int_0^\infty \int_0^\infty \sum_I \sum_J \bigl( \chi_{\{|z|<2R\}} + \chi_{\{|w|<2R\}} \bigr) \int_{I\times J} k(s,s') \ \ds\ds'\ \textup{d} \tau \dt \textup{d}w \textup{d}z \nonumber\\
        & \quad = \frac{CM}{R} \int_{\R^N} \int_{\R^N} \bigl( \chi_{T_{e,2R}}(x) + \chi_{T_{e,2R}}(y) \bigr)  f(x) g(y) K(x-y) \ \dx\dy\\
        & \quad = \frac{CM}{R} \left( \int_{T_{e,2R}} u^{q+1} (x) \ \dx + \int_{T_{e,2R}} v^{p+1}(x) \ \dx \right). \nonumber
    \end{align}
    We now average \eqref{E-oR-6} over $e\in\mathbb S^{N-1}$. By \eqref{zero-defect-0}, for $x\ne0$,
    \begin{align*}
        \frac{1}{|\mathbb{S}^{N-1}|} \int_{\mathbb{S}^{N-1}} \chi_{T_{e,2R}}(x)\ \textup{d} S(e) \leq C\min\left\{1,\left(\frac{R}{|x|}\right)^{N-1}\right\}.
    \end{align*}
    Denote $H = u^{q+1} + v^{p+1}$. Then, from \eqref{E-oR-6}, we obtain
    \begin{align*}
        \frac{1}{|\mathbb{S}^{N-1}|} \int_{\mathbb{S}^{N-1}} A_{e,R} \ \textup{d} S(e) &  \leq \frac{CM}{R} \int_{\R^N} H(x) \left( \frac{1}{|\mathbb{S}^{N-1}|} \int_{\mathbb{S}^{N-1}} \chi_{T_{e,2R}} (x) \ \textup{d} S(e) \right) \ \dx \\
        & \leq \frac{CME(R)}{R} + \frac{CM}{R} \sum_{ k = 0 }^{\infty} \int_{ 2^k R < |x| < 2^{k+1}R} H(x) \min\left\{ 1, \left(  \frac{R}{|x|} \right)^{N-1}  \right\} \ \dx \\
        & \leq \frac{CME(R)}{R} + \frac{CM}{R} \sum_{k=0}^\infty 2^{-k(N-1)} E(2^{k+2}R) \\
        & \leq C M \varepsilon(R) \left( 1 + \sum_{k=0}^{\infty} 2^{-k(N-2)} \right) \to 0,
    \end{align*}
    where $\varepsilon(R)=\sup_{r\geq R} E(r)/r \to 0 $ as $R \to +\infty$. For every $e \in \mathcal E$, $A_{e,R}$ is nonnegative and nondecreasing in $R$. By the monotone convergence theorem,
    \begin{align*}
        \int_{\mathcal E} \lim_{R\to\infty} A_{e,R} \ \textup{d} S(e) = \lim_{R\to\infty}\int_{\mathcal E}A_{e,R}\ \textup{d} S(e) = 0.
    \end{align*}
   Since $\zeta_R \to 1$ as $ R \to +\infty $, for almost every $ e \in \mathcal E$ we have
    \begin{align*}
        \Delta_e^M(I,J) = -\bigl( \vartheta_M (m_I)-\vartheta_M (m_J)\bigr) T_e(I,J)=0,
    \end{align*}
    for almost every interval pair $(I,J)$. Taking the intersection of these full-measure subsets of $\mathbb{S}^{N-1} $ over positive integers $M$ and letting $ M \to \infty $, we obtain $\Delta_e(I,J)=0$ for almost every direction $e \in \mathbb{S}^{N-1} $ and interval pair $(I, J)$. Thus the definition of $ D(\Omega) $ gives \eqref{E-oR}. Hence, Theorem \ref{thm:zero-defect} implies that $(u,v)\in\mathcal G$.
\end{proof}

\section{Proof of Theorem \ref{thm:main}}\label{sec:classification}

We prove Theorem \ref{thm:main}. First, we show that a bounded normalized solution outside $\mathcal G$ satisfies $D(B_R) \geq \delta_*>0$ for a suitable $R>0$.

Define
\begin{align*}
    m(x)=u(x)^{1/\alpha}+v(x)^{1/\beta}.
\end{align*}
Let
\begin{align*}
    \widetilde u(x)=\lambda^\alpha u(a+\lambda x),\ \ \widetilde v(x)=\lambda^\beta v(a+\lambda x),
\end{align*}
then $ \widetilde m(x)=\lambda m(a+\lambda x) $. Here $ \widetilde{m} $ is defined as $ \widetilde{m} = \widetilde{u}^{1/\alpha} + \widetilde{v}^{1/\beta} $.

\begin{thm}\label{thm:defect-gap}
    There exist constants $L_*\geq1$ and $\delta_*>0$, depending only on $N, p , q $, such that for every positive solution $ (u,v) \notin \mathcal G$, if
    \begin{align*}
        m(0)=1,  \ \  m(x) \leq 2 \ \ \text{in }B_{L_*},
    \end{align*}
    then
    \begin{align*}
        D(B_{L_*}) \geq \delta_*.
    \end{align*}
\end{thm}

\begin{proof}
    We prove this by contradiction.
    Suppose that the conclusion is false. Then, for every positive integer $j$, there exists a positive entire solution $(u_j,v_j)\notin\mathcal G$ such that
    \begin{align*}
        m_j(0)=1, \ \  m_j(x) \leq 2  \ \ \text{in }B_j, \ \ D_{u_j,v_j} (B_j)< \frac{1}{j},
    \end{align*}
    where $ m_j=u_j^{1/\alpha}+v_j^{1/\beta} $.

    We first show that this sequence has a locally convergent subsequence. Fix $R>0$. For every sufficiently large $j$, the bound $m_j\leq2$ on $B_j$ gives
    \begin{align*}
        \|u_j\|_{L^\infty(B_{2R})}\leq2^\alpha, \ \  \|v_j\|_{L^\infty(B_{2R})}\leq2^\beta.
    \end{align*}
    Since
    \begin{align*}
        -\Delta u_j=v_j^p, \ \ -\Delta v_j=u_j^q,
    \end{align*}
    interior $W^{2,r}$ estimates with $r>N$ give uniform local $C^{1,\gamma}$ bounds. The nonlinearities $s^p$ and $s^q$ are H\"older continuous on bounded subsets of $[0,\infty)$. Thus Schauder estimates give uniform local $C^{2,\eta}$ bounds for some $\eta>0$. By taking a diagonal subsequence, we obtain $(u_j,v_j)\to(u,v)$ in $C^2_{\mathrm{loc}}(\R^N)$. The limit is a nonnegative classical solution of \eqref{Main1}. Since $m_j(0)=1$, we have
    \begin{align*}
        u(0)^{1/\alpha}+v(0)^{1/\beta}=1.
    \end{align*}
    Thus $(u,v)$ is nontrivial. The strong maximum principle implies that $u$ and $v$ are positive in $ \R^N $.

    Next, we prove that $D$ is lower semicontinuous and that $D_{u,v}(B_R)=0$. Write
    \begin{align*}
        f_j = u_j^q, \ \  g_j = v_j^p, \ \  f=u^q, \ \  g=v^p.
    \end{align*}
    Then $f_j\to f$ and $g_j\to g$ in $C^1_{\mathrm{loc}}(\R^N)$, since the positive limits $u,v$ are bounded away from zero on compact sets. Let $\Omega\subset\R^N$ be a bounded open set, and fix a direction $e$, $z\in e^\perp$, and $t>0$. Let $I=(s_-,s_+)\in\mathcal I_f(e,z,t)$ be a bounded component with $z+\overline I e\subset\Omega$.

    We first assume that both endpoints are regular in the sense that
    \begin{align*}
        \partial_e f(z+s_-e) \neq 0, \ \ \partial_e f(z+s_+e) \neq 0.
    \end{align*}
    For almost every $t>0$, this assumption holds. Indeed, the set of critical values of a one-dimensional $C^1$ function has Lebesgue measure zero. Since $(s_-,s_+)$ is a connected component of $\{s:f(z+se)>t\}$, we have
    \begin{align}\label{compact-1}
        \partial_e f(z+s_-e) > 0, \ \  \partial_e f(z+s_+e)< 0.
    \end{align}
    Since $f_j \to f$ in $C^1$ near $s_\pm$, for all sufficiently large $j$, there exist unique points $s_{-,j}$ and $s_{+,j}$, such that $s_{-,j} \to s_-, \ s_{+,j}\to s_+$ and
    \begin{align}\label{compact-2}
        f_j(z+s_{-,j}e)=t, \ \  f_j(z+s_{+,j}e)=t.
    \end{align}
    Since $f-t$ is strictly positive on every compact subset of $(s_-,s_+)$, the same holds for $f_j-t$ for large $j$. By \eqref{compact-1} and $C^1$ convergence, $ f_j < t $ on the left of $s_{-,j}$ and on the right of $s_{+,j}$. Thus \eqref{compact-2} implies that
    \begin{align*}
        I_j=(s_{-,j},s_{+,j})
    \end{align*}
    is a connected component of $ \{s:f_j(z+se)>t\}$. Since $z+\overline I e$ is compactly contained in $\Omega$, the convergence of $s_{\pm,j} \to s_{\pm}$ then implies $ z + \overline{I_j}e \subset \Omega $ for every sufficiently large $j$.

    Similarly, a bounded component $J$ of $\{s:g(w+se)>\tau\}$ with regular endpoints and $w+\overline J e\subset\Omega$ gives components $J_j$ of the superlevel set of $g_j$. Thus $m_{I_j}\to m_I$ and $m_{J_j}\to m_J$ as $j\to\infty$.

    For $z\neq w$, we have $|x-y|\geq|z-w|$, so $K(x-y)$ is smooth along these intervals. Thus $T_e(I_j,J_j)\to T_e(I,J)$ and
    \begin{align*}
        \Delta_e(I_j,J_j) = -(m_{I_j} - m_{J_j}) T_e (I_j,J_j) \to  -(m_I-m_J) T_e(I,J) = \Delta_e(I,J).
    \end{align*}

    Fix finitely many distinct complete interval pairs $(I^k,J^k)$ for $(f,g)$, $k=1,\ldots,M$, with $z+\overline{I^k}e,\ w+\overline{J^k}e\subset\Omega$. For large $j$, the corresponding pairs $(I_j^k,J_j^k)$ are distinct, satisfy $z+\overline{I_j^k}e,\ w+\overline{J_j^k}e\subset\Omega$, and obey $\Delta_e(I_j^k,J_j^k)\to\Delta_e(I^k,J^k)$. Hence
    \begin{align}\label{compact-3}
        \sum_{k=1}^M \Delta_e (I^k,J^k)  = \lim_{j\to\infty} \sum_{k=1}^M \Delta_e (I_j^k,J_j^k)  \leq \liminf_{j\to\infty} \sum_{I_j\in\mathcal I_{f_j}} \sum_{J_j\in\mathcal I_{g_j}} \chi_{\{z+\overline{I_j}e\subset\Omega,\, w + \overline{J_j} e \subset \Omega\}} \Delta_e (I_j,J_j).
    \end{align}
    Taking the supremum in \eqref{compact-3} over all finite collections of interval pairs with $z+\overline I e,\ w+\overline J e\subset\Omega$, we obtain
    \begin{align}\label{compact-4}
        & \sum_{I \in \mathcal I_f(e,z,t)} \sum_{J \in \mathcal I_g(e,w,\tau)} \chi_{\{ z + \overline I e \subset \Omega, \; w + \overline J e \subset \Omega\}} \Delta_e(I,J) \\
        & \leq\liminf_{j\to\infty} \sum_{I_j\in\mathcal I_{f_j}(e,z,t)} \sum_{J_j\in\mathcal I_{g_j}(e,w,\tau)} \chi_{\{z+\overline{I_j} e \subset \Omega, \; w+\overline{J_j} e \subset \Omega\}} \Delta_e(I_j,J_j). \nonumber
    \end{align}
   Integrating \eqref{compact-4} with respect to $z,w,t,\tau$, and averaging over $e \in \mathbb{S}^{N-1} $, and applying Fatou's lemma, we obtain
    \begin{align}\label{compact-5}
        D_{u,v}(\Omega) \leq \liminf_{j\to\infty} D_{u_j,v_j} (\Omega).
    \end{align}

    For each fixed $R>0$, \eqref{compact-5} and $B_R\subset B_j$ for large $j$ give
    \begin{align*}
        0 \leq D_{u,v}(B_R) \leq \liminf_{j\to\infty}D_{u_j,v_j}(B_R) \leq \liminf_{j\to\infty}D_{u_j,v_j}(B_j)=0.
    \end{align*}
    Thus $D_{u,v}(B_R)=0$ for every $R>0$. Theorem \ref{thm:zero-defect} then implies $ (u,v)\in\mathcal G $. Without loss of generality, we may assume $(u,v) = (U,V)$. Indeed, if
    \begin{align*}
        u(x)=\lambda_0^\alpha U(\lambda_0(x-a_0)), \ \ v(x)=\lambda_0^\beta V(\lambda_0(x-a_0)),
    \end{align*}
    for some $ a_0 \in \R^N$ and $ \lambda_0 > 0$, then the sequence $(\widehat u_j,\widehat v_j)$ defined by
    \begin{align*}
        \widehat u_j(x) = \lambda_0^{-\alpha} u_j (a_0 + x/\lambda_0), \ \  \widehat v_j(x) = \lambda_0^{-\beta} v_j (a_0 + x/\lambda_0),
    \end{align*}
    is a sequence of positive solutions of \eqref{Main1} outside $\mathcal G$, and $ (\widehat u_j,\widehat v_j) \to (U,V) $ in $ C^2_{\mathrm{loc}} (\R^N) $. Relabeling this sequence, we may assume $(u_j,v_j)\to(U,V)$ in $C^2_{\mathrm{loc}}(\R^N)$.

     Recall that $U$ is radial and strictly decreasing with respect to $ r = |x|$, and $-\Delta U=V^p$ implies
    \begin{align*}
        D^2U(0)=-\frac{V(0)^p}{N}\operatorname{Id},
    \end{align*}
    so the origin is a strict nondegenerate maximum point of $U$.

    Choose a sufficiently small ball $B_r$ centered at the origin. By local $C^2$ convergence, $u_j$ attains an interior maximum at some $a_j\in B_r$ for large $j$, and
    \begin{align*}
        a_j \to 0, \ \  \nabla u_j(a_j) = 0, \ \ u_j(a_j) \to U(0) = 1.
    \end{align*}
    Set $ \lambda_j = u_j(a_j)^{-1/\alpha} $ and define
    \begin{align}\label{compact-6}
        \widetilde u_j(x) = \lambda_j^\alpha u_j (a_j + \lambda_j x), \ \  \widetilde v_j(x) = \lambda_j^\beta v_j (a_j+  \lambda_jx).
    \end{align}
    Then
    \begin{align}\label{compact-7}
        \widetilde u_j(0) = 1 , \ \  \nabla \widetilde u_j (0) = 0.
    \end{align}
    Furthermore, since $ a_j \to 0 $ and $ \lambda_j \to 1 $, we have
    \begin{align}\label{compact-8}
        (\widetilde u_j,\widetilde v_j) \to (U,V) \ \ \text{in } \ \ C^2_{\mathrm{loc}}(\R^N).
    \end{align}
    By Theorem \ref{Isolate-thm}, \eqref{compact-7}, and \eqref{compact-8}, we have
    \begin{align*}
        (\widetilde u_j,\widetilde v_j)=(U,V).
    \end{align*}
    From \eqref{compact-6}, we conclude that  $ (u_j,v_j)\in\mathcal G $ for all sufficiently large $j$. This contradicts $(u_j,v_j)\notin\mathcal G$ and proves the theorem.
\end{proof}

Next, in order to prove Theorem \ref{thm:main}, we introduce the doubling lemma of Pol\'a\v{c}ik, Quittner, and Souplet \cite[Lemma 5.1, Remark 5.2]{PQS2007}.

\begin{lemma}[Doubling lemma \cite{PQS2007}]\label{lem:doubling}
    Let $\Omega\subset\R^N$ be open, and let $M:\Omega\to(0,\infty)$ be locally bounded. Suppose that $k>0$ and $x_0\in\Omega$ satisfy
    \begin{align*}
        M(x_0)\operatorname{dist}(x_0,\partial\Omega)>2k.
    \end{align*}
    Then there exists $y\in\Omega$ such that
    \begin{align*}
        M(y)\geq M(x_0), \ \  M(y)\operatorname{dist}(y,\partial\Omega)>2k,
    \end{align*}
    and
    \begin{align*}
        M(x) \leq 2M(y), \ \ x\in B_{k/M(y)}(y).
    \end{align*}
    When $\Omega=\R^N$, we set $\operatorname{dist}(x,\partial\Omega)=\infty$.
\end{lemma}

Theorem \ref{thm:defect-gap} and Lemma \ref{lem:doubling} give the following lemma.

\begin{lemma}\label{D->m}
    Let $(u,v)$ be a positive solution of \eqref{Main1} with $(u,v)\notin\mathcal G$. For every $a\in\R^N$ and $r>0$, if $D(B_{2r}(a))<\delta_*$, then
    \begin{align*}
        \sup_{B_r(a)} m \leq\frac{8L_*}{r},
    \end{align*}
    where $\delta_*$ and $ L_* $ are the constants in Theorem \ref{thm:defect-gap} and $ m(x) = u(x)^{1/\alpha}+v(x)^{1/\beta} $.
\end{lemma}

\begin{proof}
    Suppose, to the contrary, that $D(B_{2r}(a)) < \delta_*$, but there exists $x_0 \in B_r(a)$ such that
    \begin{align*}
        m(x_0)>\frac{8 L_*}{r}.
    \end{align*}
    Since $x_0\in B_r(a)$, we have
    \begin{align*}
        m(x_0) \operatorname{dist}(x_0, \partial B_{2r}(a)) > 8L_* > 2k.
    \end{align*}
    Apply Lemma \ref{lem:doubling} with $\Omega=B_{2r}(a)$, $M=m$, and $k=L_*$, there is a point $ y \in B_{2r}(a) $ such that
    \begin{align}\label{D->m-1}
        m(y)\geq m(x_0),\ \  m(y)\operatorname{dist}(y,\partial B_{2r}(a))>2k,
    \end{align}
    and
    \begin{align}\label{D->m-2}
        m(x)\leq2m(y),\ \  x\in B_{L_*/m(y)}(y).
    \end{align}
    By \eqref{D->m-1},
    \begin{align*}
        |y-a|+\frac{L_*}{m(y)}
        &=2r-\operatorname{dist}(y,\partial B_{2r}(a))+\frac{L_*}{m(y)}\\
        &<2r+\frac{L_*-2k}{m(y)}<2r.
    \end{align*}
    Hence $B_{L_*/m(y)}(y)\subset B_{2r}(a)$. Set $\lambda=m(y)^{-1}$ and define
    \begin{align*}
        \widetilde u(x) = \lambda^\alpha u(y+\lambda x), \ \ \widetilde v(x)=\lambda^\beta v( y+ \lambda x).
    \end{align*}
    Then $(\widetilde u,\widetilde v)$ is a solution of \eqref{Main1} outside $\mathcal G$, and
    \begin{align*}
        \widetilde m(0)=1, \ \  \widetilde m \leq 2, \ \ \text{in } \ \ B_{L_*}.
    \end{align*}
    Theorem~\ref{thm:defect-gap} and the scaling property in Lemma~\ref{lem:defect-properties} give
    \begin{align*}
        \delta_* & \leq D_{\widetilde u,\widetilde v}(B_{L_*}) =D_{u,v}(B_{L_*/m(y)}(y)) \leq D_{u,v}(B_{2r}(a)) < \delta_*,
    \end{align*}
    a contradiction. This completes the proof.
\end{proof}

Next, we prove Theorem \ref{thm:main} under the assumption that $u$ and $v$ are bounded.

\begin{thm}\label{thm-bounded}
    Let $(u,v)$ be a positive bounded solution of \eqref{Main1}. Then $(u,v)\in\mathcal G$.
\end{thm}

\begin{proof}
    Suppose, to the contrary, that $(u,v) \notin \mathcal G$. Since $ m $ is positive and bounded, without loss of generality, we may assume $ \sup_{\R^N}m=1 $. This follows by dilating $(u,v)$ and keeping the same notation. Fix $ R > 16 L_* $ and $ (16 L_*)/R\leq t\leq 1$. Set
    \begin{align}\label{bounded-main-1}
        \Omega_t = \{ x \in B_R: m(x) > t \}, \ \  r=\frac{8 L_*}{t} \leq\frac{R}{2}.
    \end{align}
    For every $ x \in \Omega_t $, we have $ m(x) > 8 L_*/r $. Then Lemma \ref{D->m} implies
    \begin{align*}
        D(B_{2r}(x)) \geq \delta_*.
    \end{align*}
    Since $x\in B_R$ and $2r\leq R$, each ball $B_{2r}(x)$ lies in $B_{2R}$. Choose a maximal pairwise disjoint family $B_{2r}(x_1),\ldots,B_{2r}(x_k)$ with $x_i\in\Omega_t$. Then $\Omega_t\subset\bigcup_{i=1}^k B_{4r}(x_i)$. By Lemma \ref{lem:defect-properties},
    \begin{align}\label{bounded-main-2}
        k \delta_* \leq \sum_{i=1}^k D(B_{2r}(x_i)) \leq D_{2R}.
    \end{align}
    It follows from \eqref{bounded-main-1} and \eqref{bounded-main-2} that
    \begin{align}\label{bounded-main-3}
        |\Omega_t| \leq \sum_{i=1}^k |B_{4r}(x_i)| \leq C k r^N \leq C t^{-N} D_{2R},
    \end{align}
    where $C$ depends on $L_*$ and $\delta_*$.

    By the definition of $m$, we have $u^{q+1}+v^{p+1}\leq m^N$. Take $ t_0 \in (0,1) $, we obtain
    \begin{align*}
        E(R) & \leq \int_{B_R} m^N \ \dx = N \int_0^1 t^{N-1}  |\Omega_t|\ \dt \\
        & \leq N |B_R| \int_0^{t_0} t^{N-1} \ \dt + C D_{2R} \int_{t_0}^1 \frac{\dt}{t} \\
        & = |B_1| R^N t_0^N + C D_{2R} \log \frac{1}{t_0}.
    \end{align*}
    For $ R > 16 L_* $, we choose $t_0 = 16L_*/R$ and obtain
    \begin{align}\label{bounded-main-4}
        E(R) \leq C+C\log(2+R)D_{2R}.
    \end{align}
    Fix
    \begin{align*}
        0 < \sigma < \min\{1, \alpha, \beta\}, \ \ \theta = 1- \frac{\sigma}{N} \in(0,1).
    \end{align*}
    Applying Proposition \ref{prop:defect-energy} at radius $2R$, we obtain
    \begin{align*}
        D_{2R} \leq C( E(4R)^{1-\sigma/N} + E(4R)^{1-1/N} ) \leq C\bigl(1+E(4R)\bigr)^\theta,
    \end{align*}
    which, together with \eqref{bounded-main-4}, gives
    \begin{align}\label{bounded-main-5}
        1+E(R) \leq C \log(2+R) \bigl( 1+E(4R) \bigr)^\theta,
    \end{align}
    when $ R > 16L_* $. For $ 1 \leq R \leq 16 L_*$, the same estimate follows after increasing $C$, since $ u^{q+1} + v^{p+1} \leq m^N\leq1$. Thus, \eqref{bounded-main-5} holds for all $ R \geq 1 $.

    Set $ F(R) = 1 + E(R) $. Then $ m \leq 1$ implies that
    \begin{align}\label{bounded-main-6}
        1 \leq F(R) \leq1+|B_R|\leq C(1+R^N), \ \ \text{for all} \ \ R\geq1,
    \end{align}
    and \eqref{bounded-main-5} implies
    \begin{align}\label{bounded-main-7}
        F(R)\leq C\log(2+R)F(4R)^\theta.
    \end{align}
    Taking logarithms in \eqref{bounded-main-7}, we obtain
    \begin{align}\label{bounded-main-8}
        \log F(R) \leq \log C + \log \log (2+R) + \theta \log F(4R).
    \end{align}
    Iterating \eqref{bounded-main-8} $k$ times, we obtain
    \begin{align}\label{bounded-main-9}
        \log F(R) \leq \sum_{j=0}^{k-1} \theta^j \bigl( \log C + \log\log( 2 + 4^jR ) \bigr) + \theta^k \log F(4^k R).
    \end{align}
    From \eqref{bounded-main-6}, we deduce for every fixed $ R \geq 1 $ that
    \begin{align}\label{bounded-main-10}
        0 \leq \theta^k \log F(4^kR) \leq \theta^k \bigl( C + N \log R + k N \log4 \bigr) \to 0,
    \end{align}
    as $k\to\infty$. Moreover,
    \begin{align}\label{bounded-main-11}
        \log \log(2+4^jR) & \leq \log\bigl( j \log4 + \log(2+R) \bigr) \nonumber \\
        & \leq \log \log(2+R) + \log(1+j)+C.
    \end{align}
    Since $\sum_{j\geq0}\theta^j\log(1+j)<\infty$, letting $k\to\infty$ in \eqref{bounded-main-9} and using \eqref{bounded-main-10}--\eqref{bounded-main-11}, we obtain
    \begin{align*}
        \log F(R)\leq\frac1{1-\theta}\log\log(2+R)+C.
    \end{align*}
    Therefore,
    \begin{align*}
        E(R) \leq F(R) \leq C\bigl( \log(2+R)\bigr)^{1/(1-\theta)} = C \bigl( \log(2+R) \bigr)^{N/\sigma} = o(R).
    \end{align*}
    Theorem \ref{thm:sublinear-energy} gives $(u,v)\in\mathcal G$, a contradiction. This completes the proof.
\end{proof}

Finally, we prove Theorem \ref{thm:main} by removing the boundedness assumption in Theorem \ref{thm-bounded}.

\begin{proof}[Proof of the main theorem]
    Let $(u,v)$ be a positive solution of \eqref{Main1}. Suppose, to the contrary, that $(u,v)\notin\mathcal G$. We use Lemma \ref{lem:doubling} to obtain a bounded positive solution and then apply Theorem \ref{thm-bounded}.

    Fix $x_0\in\R^N$. For each positive integer $j$, apply Lemma \ref{lem:doubling} with $\Omega=\R^N$, $M=m$, and $k=j$. We obtain a point $y_j$ such that
    \begin{align*}
        m(x) \leq 2 m(y_j),\ \ |x-y_j|<\frac{j}{m(y_j)}.
    \end{align*}
    Write $M_j=m(y_j)$ and define
    \begin{align*}
        u_j(x)=M_j^{-\alpha}u(y_j+x/M_j), \ \  v_j(x)=M_j^{-\beta} v(y_j+x/M_j).
    \end{align*}
    Then $(u_j,v_j)$ is a positive solution of \eqref{Main1} outside $\mathcal G$. With $ m_j = u_j^{1/\alpha} + v_j^{1/\beta} = M_j^{-1}m(y_j+x/M_j)$, we have
    \begin{align*}
        m_j(0)=1,\ \  m_j(x)\leq2 , \ \ \text{in } \ \ B_j.
    \end{align*}
    On every fixed ball, $(u_j,v_j)$ is uniformly bounded for sufficiently large $j$. Interior $W^{2,r}$ estimates with $r>N$, followed by Schauder estimates, give uniform local $C^{2,\gamma}$ bounds for some $\gamma>0$. This also holds when $p<1$, since the power nonlinearities are H\"older continuous on bounded nonnegative intervals. Taking a diagonal subsequence, we obtain $(u_j,v_j) \to (u_\infty,v_\infty)$ in $C^2_{\mathrm{loc}}(\R^N)$, and the limit $(u_{\infty}, v_{\infty})$ is also a nonnegative solution of \eqref{Main1}, with
    \begin{align}\label{prove-bounded-1}
        u_\infty(0)^{1/\alpha} + v_\infty(0)^{1/\beta}=1, \ \ u_\infty(x)^{1/\alpha} + v_\infty(x)^{1/\beta} \leq2 \quad \text{in }\R^N.
    \end{align}
    From the strong maximum principle, we know that $u_{\infty}$ and $v_{\infty}$ are positive. Theorem \ref{thm-bounded} then gives $ (u_\infty,v_\infty ) \in \mathcal G$. After applying translation and dilation to $(u_j,v_j)$, and keeping the same notation, we may assume that $(u_j,v_j)\to(U,V)$ in $C^2_{\mathrm{loc}}(\R^N)$.

    As at the end of the proof of Theorem \ref{thm:defect-gap}, for sufficiently large $j$ there is an interior maximum point $a_j$ of $u_j$ in a fixed neighborhood of the origin, with
    \begin{align}\label{prove-bounded-2}
        a_j \to 0, \ \ \nabla u_j(a_j)=0, \ \  u_j(a_j)\to U(0)=1.
    \end{align}
    Set $\lambda_j=u_j(a_j)^{-1/\alpha}$ and define
    \begin{align*}
        \widehat u_j(x)=\lambda_j^\alpha u_j(a_j+\lambda_jx), \ \ \widehat v_j(x)=\lambda_j^\beta v_j(a_j+\lambda_jx).
    \end{align*}
    Then $ \widehat u_j(0)=1,  \nabla\widehat u_j(0)=0 $,  and \eqref{prove-bounded-2} gives
    \begin{align*}
        (\widehat u_j,\widehat v_j) \to (U,V) \ \ \text{in }C^2_{\mathrm{loc}}(\R^N).
    \end{align*}
    By Theorem \ref{Isolate-thm}, $(\widehat u_j,\widehat v_j)=(U,V)$ for all sufficiently large $j$. From the definition of $ (\widehat u_j, \widehat v_j) $ and $ (u_j, v_j )$, we get $ (u, v ) \in \mathcal G $, a contradiction. Thus every positive solution of \eqref{Main1} belongs to $\mathcal G$, and there exist $a\in\R^N$ and $\lambda>0$ such that
    \begin{align*}
        u(x)=\lambda^\alpha U(\lambda(x-a)), \ \
        v(x)=\lambda^\beta V(\lambda(x-a)).
    \end{align*}
    In fact, $ a $ is the unique maximal point of $u$ and $ \lambda = u(a)^{1/\alpha}$. Since $\alpha(q+1)=\beta(p+1)=N$, we have
    \begin{align*}
        \int_{\R^N}u^{q+1}\,\dx &= \lambda^{\alpha(q+1)-N}\int_{\R^N}U^{q+1}\,\dx=E_*,\\
        \int_{\R^N}v^{p+1}\,\dx &=\lambda^{\beta(p+1)-N}\int_{\R^N}V^{p+1}\,\dx=E_*.
    \end{align*}
     This completes the proof.
\end{proof}

\medskip
\noindent\textbf{Data availability statement.}
No datasets were generated or analyzed in this study.

\medskip

\noindent\textbf{Declaration on the use of AI tools.}
The authors used ChatGPT to assist in exploring research ideas and possible proof strategies, and to improve the language of the Introduction. All mathematical arguments and proofs were carefully checked and verified by the authors. The manuscript was written and finalized by the authors, who take full responsibility for its content.

\end{document}